\documentclass{amsart}
\usepackage{setspace}
\usepackage{a4}
\usepackage{amssymb,amsmath,amsthm,latexsym}
\usepackage{amsfonts}
\usepackage{amsfonts}
\usepackage{graphicx}
\usepackage{textcomp}
\usepackage{cite}
\usepackage{enumerate}
\usepackage[mathscr]{euscript}
\usepackage{mathtools}
\newtheorem{theorem}{Theorem}[section]

\newtheorem{conjecture}[theorem]{Conjecture}
\newtheorem{corollary}[theorem] {Corollary}
\newtheorem{definition}[theorem]{Definition}

\newtheorem{proposition}[theorem]{Proposition}

\newtheorem{question}[theorem]{Question}
\title{This is the title}
\usepackage{amssymb}
\usepackage{amssymb}
\usepackage{amssymb}
\usepackage{amssymb}
\usepackage{amsmath}
\usepackage{tikz}
\usepackage{hyperref}
\usepackage{enumerate}
\usepackage{mathtools}
\usepackage{amsmath}
\usepackage{tikz}
\usepackage{amssymb}
\usepackage{amsmath}
\usepackage{tikz}
\usepackage{hyperref}
\usepackage{fancyhdr}
\begin{document}
\begin{center}
{\bf{DISCRETE AND CONTINUOUS WELCH BOUNDS  FOR  BANACH SPACES WITH APPLICATIONS}}\\
\textbf{K. MAHESH KRISHNA}\\
Post Doctoral Fellow \\
Statistics and Mathematics Unit\\
Indian Statistical Institute, Bangalore Centre\\
Karnataka 560 059, India\\
Email: kmaheshak@gmail.com\\

Date: \today
\end{center}
\hrule
\vspace{0.5cm}
\textbf{Abstract}: 	Let $\{\tau_j\}_{j=1}^n$ be a collection in a finite dimensional Banach space $\mathcal{X}$ of dimension $d$  and 	$\{f_j\}_{j=1}^n$ be a collection in  $\mathcal{X}^*$ (dual of $\mathcal{X}$) such that 	$f_j(\tau_j) =1$, $\forall 1\leq j\leq n$. 	Let $n\geq d$ and   $\text{Sym}^m(\mathcal{X})$ be  the Banach  space of symmetric m-tensors.   If the  operator $	\text{Sym}^m(\mathcal{X})\ni x \mapsto  \sum_{j=1}^nf_j^{\otimes m}(x)\tau_j ^{\otimes m}\in\text{Sym}^m(\mathcal{X})$	 is diagonalizable and its eigenvalues are all non negative, then  we prove that 
\begin{align}\label{WELCHBANACHABSTRACT}
\max _{1\leq j,k \leq n, j\neq k}|f_j(\tau_k)|^{2m}\geq \max _{1\leq j,k \leq n, j\neq k}|f_j(\tau_k)f_k(\tau_j)|^m \geq\frac{1}{n-1}\left[\frac{n}{{d+m-1\choose m}}-1\right], \quad \forall m \in \mathbb{N}.
\end{align}
When $	\mathcal{X}=\mathcal{H}$  is a Hilbert space, and $f_j$  is defined by $f_j: \mathcal{H}\ni h \mapsto \langle h, \tau_j \rangle \in \mathbb{K}$ (where $\mathbb{K}$ is $\mathbb{R}$ or $\mathbb{C}$), $\forall 1 \leq j \leq n$, then Inequality (\ref{WELCHBANACHABSTRACT}) reduces to Welch bounds. Thus Inequality (\ref{WELCHBANACHABSTRACT})   improves 48 years old result  obtained by Welch [\textit{IEEE Transactions on  Information Theory, 1974}]. We also prove the following continuous version of Inequality (\ref{WELCHBANACHABSTRACT}). Let $m \in \mathbb{N}$, $(\Omega, \mu)$ be a $\sigma$-finite measure space and $(\{f_\alpha\}_{\alpha\in \Omega}, \{\tau_\alpha\}_{\alpha\in \Omega})$ be a  continuous Bessel  family for a finite dimensional Banach   space $\mathcal{X}$ of dimension $d$ such that $f_\alpha(\tau_\alpha)=1$, $\forall \alpha \in \Omega$.  If the diagonal $\Delta\coloneqq \{(\alpha, \alpha):\alpha \in \Omega\}$ is measurable in the measure space $\Omega\times \Omega$,
\begin{align*}
	\int_{\Omega\times\Omega}|f_\alpha(\tau_\beta)f_\beta(\tau_\alpha)|^m\, d(\mu\times\mu)(\alpha,\beta)<\infty,
\end{align*} 
and the operator $	\text{Sym}^m(\mathcal{X})\ni x \mapsto  \int_{\Omega}f_\alpha^{\otimes m}(x)\tau_\alpha ^{\otimes m}\, d\mu (\alpha)\in\text{Sym}^m(\mathcal{X})$	 is diagonalizable and its eigenvalues are all non negative, then  we prove that  
\begin{align}\label{CONTINUOUSWELCHBANACHABSTRACT}
	\sup _{\alpha, \beta \in \Omega, \alpha\neq \beta}|f_\alpha(\tau_\beta) |^{2m}\geq 	\sup _{\alpha, \beta \in \Omega, \alpha\neq \beta}|f_\alpha(\tau_\beta)f_\beta(\tau_\alpha) |^{m}\geq \frac{1}{(\mu\times\mu)((\Omega\times\Omega)\setminus\Delta)}\left[\frac{	\mu(\Omega)^2}{{d+m-1 \choose m}}-(\mu\times\mu)(\Delta)\right].
\end{align}
Inequality (\ref{CONTINUOUSWELCHBANACHABSTRACT}) improves the continuous Welch bounds for Hilbert spaces obtained by the author in the year 2021. Motivated from several high end  Hilbert space results and problems, we formulate several open problems including Zauner's conjecture for Banach spaces.



\hrule

\tableofcontents

\section{Introduction}
Given a collection $\{\tau_j\}_{j=1}^n$ of  unit vectors in $\mathbb{C}^d$, using Cauchy-Schwarz inequality we get 
\begin{align}\label{WELCHBEGINNING}
\max _{1\leq j,k \leq n, j\neq k}|\langle \tau_j, \tau_k\rangle |^{2}\leq 1.
\end{align}
The natural question which comes immediately is that whether there is any lower bound for the quantity max in Inequality (\ref{WELCHBEGINNING}). In his celebrated paper \cite{WELCH}, L. Welch proved the following result in 1974.
  \begin{theorem}\cite{WELCH}\label{WELCHTHEOREM} (\textbf{Welch bounds})
	Let $n\geq d$.	If	$\{\tau_j\}_{j=1}^n$  is any collection of  unit vectors in $\mathbb{C}^d$, then
	\begin{align*}
		\sum_{1\leq j,k \leq n}|\langle \tau_j, \tau_k\rangle |^{2m}=\sum_{j=1}^n\sum_{k=1}^n|\langle \tau_j, \tau_k\rangle |^{2m}\geq \frac{n^2}{{d+m-1\choose m}}, \quad \forall m \in \mathbb{N}.
	\end{align*}
	In particular,
	\begin{align*}
	\sum_{1\leq j,k \leq n}|\langle \tau_j, \tau_k\rangle |^{2}=		\sum_{j=1}^n\sum_{k=1}^n|\langle \tau_j, \tau_k\rangle |^{2}\geq \frac{n^2}{{d}}.
	\end{align*}
	Further, 
	\begin{align*}
		\text{(\textbf{Higher order Welch bounds})}	\quad		\max _{1\leq j,k \leq n, j\neq k}|\langle \tau_j, \tau_k\rangle |^{2m}\geq \frac{1}{n-1}\left[\frac{n}{{d+m-1\choose m}}-1\right], \quad \forall m \in \mathbb{N}.
	\end{align*}
	In particular,
	\begin{align*}
		\text{(\textbf{First order Welch bound})}\quad 	\max _{1\leq j,k \leq n, j\neq k}|\langle \tau_j, \tau_k\rangle |^{2}\geq\frac{n-d}{d(n-1)}.
	\end{align*}
\end{theorem}
 There are several theoretical and practical applications of Theorem \ref{WELCHTHEOREM} such as in the study of root-mean-square (RMS) absolute cross relation of unit vectors   \cite{SARWATEMEETING}, frame potential \cite{BENEDETTOFICKUS, CASAZZAFICKUSOTHERS, BODMANNHAASPOTENTIAL}, 
 correlations \cite{SARWATE},  codebooks \cite{DINGFENG}, numerical search algorithms  \cite{XIA, XIACORRECTION}, quantum measurements 
\cite{SCOTTTIGHT}, coding and communications \cite{TROPPDHILLON, STROHMERHEATH}, code division multiple access (CDMA) systems \cite{CHEBIRA1, CHEBIRA2}, wireless systems \cite{YATES}, compressed sensing \cite{TAN}, `game of Sloanes' \cite{JASPERKINGMIXON}, equiangular tight frames \cite{SUSTIKTROPP},  etc.

Some improvements of Theorem \ref{WELCHTHEOREM}  has been done in \cite{CHRISTENSENDATTAKIM, DATTAWELCHLMA, WALDRONSH, WALDRON2003}.  It is in the paper \cite{DATTAHOWARD} where the following generalization of Theorem \ref{WELCHTHEOREM} has been done for continuous collections.
 \begin{theorem}\cite{DATTAHOWARD}\label{DATTATHEOREM}
	Let $\mathbb{C}	\mathbb{P}^{n-1}$ be the complex projective space and $\mu$ be a normalized measure on $\mathbb{C}	\mathbb{P}^{n-1}$.  If $\{\tau_\alpha\}_{\alpha \in \mathbb{C}	\mathbb{P}^{n-1}}$ is a continuous frame for a $d$-dimensional subspace $\mathcal{H}$  of a Hilbert space $\mathcal{H}_0$, then 
	\begin{align*}
	\int_{\mathbb{C}	\mathbb{P}^{n-1}\times\mathbb{C}	\mathbb{P}^{n-1}}|\langle \tau_\alpha, \tau_\beta\rangle|^{2m}\, d(\mu\times\mu)(\alpha,\beta)=		\int_{\mathbb{C}	\mathbb{P}^{n-1}}\int_{\mathbb{C}	\mathbb{P}^{n-1}}|\langle \tau_\alpha, \tau_\beta\rangle|^{2m}\, d \mu(\alpha)\, d \mu(\beta)\geq \frac{1}{{d+m-1\choose m}}, \quad \forall m \in \mathbb{N}.
	\end{align*}
\end{theorem}
Theorem \ref{DATTATHEOREM} has been recently generalized in full generality for Hilbert spaces and $\sigma$-finite  measure spaces by the author in \cite{MAHESHKRISHNA}.
\begin{theorem}\cite{MAHESHKRISHNA}\label{CONTINUOUSHILBERT}
	Let $(\Omega, \mu)$ be a  measure space and $\{\tau_\alpha\}_{\alpha\in \Omega}$ be a 	normalized continuous Bessel  family for $\mathcal{H}$ of dimension $d$. If the diagonal $\Delta\coloneqq \{(\alpha, \alpha):\alpha \in \Omega\}$ is measurable in the measure space $\Omega\times \Omega$, then 
	\begin{align*}
		\int_{\Omega\times\Omega}|\langle \tau_\alpha, \tau_\beta\rangle|^{2m}\, d(\mu\times\mu)(\alpha,\beta)=	\int_{\Omega}\int_{\Omega}|\langle \tau_\alpha, \tau_\beta\rangle|^{2m}\, d \mu(\alpha)\, d \mu(\beta)\geq \frac{\mu(\Omega)^2}{{d+m-1\choose m}}, \quad \forall m \in \mathbb{N}.
	\end{align*}	
	Further, we have the \textbf{higher order continuous Welch bounds} 
	\begin{align*}
		\sup _{\alpha, \beta \in \Omega, \alpha\neq \beta}|\langle \tau_\alpha, \tau_\beta\rangle |^{2m}\geq \frac{1}{(\mu\times\mu)((\Omega\times\Omega)\setminus\Delta)}	\left[\frac{	\mu(\Omega)^2}{{d+m-1 \choose m}}-(\mu\times\mu)(\Delta)\right],  \quad \forall m \in \mathbb{N}.
	\end{align*}
In particular, we have the \textbf{first order continuous Welch bound} 
	\begin{align*}
	\sup _{\alpha, \beta \in \Omega, \alpha\neq \beta}|\langle \tau_\alpha, \tau_\beta\rangle |^{2}\geq \frac{1}{(\mu\times\mu)((\Omega\times\Omega)\setminus\Delta)}\left[\frac{\mu(\Omega)^2}{d}-(\mu\times\mu)(\Delta)\right].
\end{align*}
\end{theorem}
Recently Theorem  \ref{WELCHTHEOREM}  and Theorem \ref{CONTINUOUSHILBERT}   are done in the context of Hilbert C*-modules by the author in \cite{MAHESHKRISHNA2}.
\begin{theorem}\cite{MAHESHKRISHNA2}
\textbf{(Modular Welch bounds)}
Let $\mathcal{A}$ be a unital commutative C*-algebra and $\mathcal{A}^d$ be the standard Hilbert C*-module over  $\mathcal{A}$.  Let $n\geq d$. If 	$\{\tau_j\}_{j=1}^n$ is any collection of unit inner product  vectors in $\mathcal{A}^d$, then 
\begin{align*}
	\sum_{j=1}^n\sum_{k=1}^n\|\langle \tau_j, \tau_k\rangle \|^{2m}	\geq \sum_{j=1}^n\sum_{k=1}^n\langle \tau_j, \tau_k\rangle ^{m}\langle \tau_k, \tau_j\rangle ^{m}\geq \frac{n^2}{{d+m-1\choose m}}, \quad \forall m \in \mathbb{N}.
\end{align*}
In particular,
\begin{align*}
	\sum_{j=1}^n\sum_{k=1}^n\|\langle \tau_j, \tau_k\rangle \|^{2} \geq 	\sum_{j=1}^n\sum_{k=1}^n\langle \tau_j, \tau_k\rangle \langle \tau_k, \tau_j\rangle \geq \frac{n^2}{{d}}.
\end{align*}
Further, 
\begin{align*}
	\text{(\textbf{Higher order modular Welch bounds})}	\quad		\max _{1\leq j,k \leq n, j\neq k}\|\langle \tau_j, \tau_k\rangle \|^{2m}\geq \frac{1}{n-1}\left[\frac{n}{{d+m-1\choose m}}-1\right], \quad \forall m \in \mathbb{N}.
\end{align*}
In particular,
\begin{align*}
	\text{(\textbf{First order modular Welch bound})}\quad 	\max _{1\leq j,k \leq n, j\neq k}\|\langle \tau_j, \tau_k\rangle \|^{2}\geq\frac{n-d}{d(n-1)}.
\end{align*}	
\end{theorem}
\begin{theorem}\cite{MAHESHKRISHNA2}
\textbf{(Modular Welch bounds)}
Let $\mathcal{A}$ be a commutative $\sigma$-finite W*-algebra or a commutative  AW*-algebra. Let $\mathcal{E}$ be  a Hilbert C*-module over $\mathcal{A}$ of rank $d$.  Let $n\geq d$. If 	$\{\tau_j\}_{j=1}^n$ is any collection of unit inner product  vectors in $\mathcal{E}$, then 
\begin{align*}
	\sum_{j=1}^n\sum_{k=1}^n\|\langle \tau_j, \tau_k\rangle \|^{2m}	\geq \sum_{j=1}^n\sum_{k=1}^n\langle \tau_j, \tau_k\rangle ^{m}\langle \tau_k, \tau_j\rangle ^{m}\geq \frac{n^2}{{d+m-1\choose m}}, \quad \forall m \in \mathbb{N}.
\end{align*}
In particular,
\begin{align*}
	\sum_{j=1}^n\sum_{k=1}^n\|\langle \tau_j, \tau_k\rangle \|^{2} \geq 	\sum_{j=1}^n\sum_{k=1}^n\langle \tau_j, \tau_k\rangle \langle \tau_k, \tau_j\rangle \geq \frac{n^2}{{d}}.
\end{align*}
Further, 
\begin{align*}
	\text{(\textbf{Higher order modular Welch bounds})}	\quad		\max _{1\leq j,k \leq n, j\neq k}\|\langle \tau_j, \tau_k\rangle \|^{2m}\geq \frac{1}{n-1}\left[\frac{n}{{d+m-1\choose m}}-1\right], \quad \forall m \in \mathbb{N}.
\end{align*}
In particular,
\begin{align*}
	\text{(\textbf{First order modular Welch bound})}\quad 	\max _{1\leq j,k \leq n, j\neq k}\|\langle \tau_j, \tau_k\rangle \|^{2}\geq\frac{n-d}{d(n-1)}.
\end{align*}	
\end{theorem}
\begin{theorem}\cite{MAHESHKRISHNA2}
(\textbf{Continuous modular Welch bounds})
Let $\mathcal{A}$ be a commutative $\sigma$-finite W*-algebra or a commutative  AW*-algebra. Let $\mathcal{E}$ be  a Hilbert C*-module over $\mathcal{A}$ of rank $d$. 		Let $(G, \mu_G)$ be a  Lie group  and $\{\tau_g\}_{g\in G}$ be a 	 continuous modular Bessel  family for $\mathcal{E}$ such that $\langle \tau_g, \tau_g \rangle =1, \forall g \in G$. If the diagonal $\Delta\coloneqq \{(g, g):g \in G\}$ is measurable in  $G\times G$, then 

\begin{align*}
	\int_{G\times G}\|\langle \tau_g, \tau_h\rangle\|^{2m}\, d(\mu_G\times\mu_G)(g,h)&=	\int_{G}\int_{G}\|\langle \tau_g, \tau_h\rangle\|^{2m}\, d \mu_G(g)\, d \mu_G(h)\\
	&\geq	\int_{G\times G}\langle \tau_g, \tau_h\rangle^{m}\langle \tau_h, \tau_g\rangle^{m}\, d(\mu_G\times\mu_G)(g,h)\\
	&=	\int_{G}\int_{G}\langle \tau_g, \tau_h\rangle^{m}\langle \tau_h, \tau_g\rangle^{m}\, d \mu_G(g)\, d \mu_G(h)\geq \frac{\mu_G(G)^2}{{d+m-1\choose m}}, \quad \forall m \in \mathbb{N}.
\end{align*}	
In particular, 
\begin{align*}
	\int_{G\times G}\|\langle \tau_g, \tau_h\rangle\|^{2}\, d(\mu_G\times\mu_G)(g,h)&=	\int_{G}\int_{G}\|\langle \tau_g, \tau_h\rangle\|^{2}\, d \mu_G(g)\, d \mu_G(h)\\
	&\geq	\int_{G\times G}\langle \tau_g, \tau_h\rangle\langle \tau_h, \tau_g\rangle\, d(\mu_G\times\mu_G)(g,h)\\
	&=	\int_{G}\int_{G}\langle \tau_g, \tau_h\rangle\langle \tau_h, \tau_g\rangle\, d \mu_G(g)\, d \mu_G(h)\geq \frac{\mu_G(G)^2}{d}.
\end{align*}
Further, we have the \textbf{higher order continuous modular Welch bounds} 
\begin{align*}
	\sup _{g, h \in G, g\neq h}\|\langle \tau_g, \tau_h\rangle \|^{2m}\geq \frac{1}{(\mu_G\times\mu_G)((G\times G)\setminus\Delta)}	\left[\frac{	\mu_G(G)^2}{{d+m-1 \choose m}}-(\mu_G\times\mu_G)(\Delta)\right],  \quad \forall m \in \mathbb{N}.
\end{align*}
In particular, we have the \textbf{first order continuous modular Welch bound} 
\begin{align*}
	\sup _{g, h \in G, g\neq h}\|\langle \tau_g, \tau_h\rangle \|^{2}\geq \frac{1}{(\mu_G\times\mu_G)((G\times G)\setminus\Delta)}\left[\frac{\mu_G(G)^2}{d}-(\mu_G\times\mu_G)(\Delta)\right].
\end{align*}	
\end{theorem}
Even after 48 years of work of Welch, to the best of our knowledge,  there are no Welch bounds or their twins or relatives  for Banach spaces. In this paper we derive both discrete and continuous Welch bounds for Banach spaces. We ask several open problems for further research.

\section{Discrete Welch bounds for Banach spaces}
Throughout the paper,  $\mathcal{X}$ denotes a finite dimensional  Banach space and $\mathcal{X}^*$ denotes its  dual. $ I_\mathcal{X}$ denotes the identity operator on $\mathcal{X}$.  We use $\mathbb{K}$ to denote $\mathbb{R}$  or $\mathbb{C}$.
 \begin{definition}
 Let 	$\mathcal{X}$ be a finite dimensional Banach space. Given a collection $\{\tau_j\}_{j=1}^n$ in $\mathcal{X}$  and 	  a collection $\{f_j\}_{j=1}^n$  in  $\mathcal{X}^*$,  
 \begin{enumerate}[\upshape(i)]
 	\item The \textbf{frame operator} is defined as 
 	\begin{align*}
 		S_{f, \tau}:\mathcal{X}\ni x \mapsto S_{f, \tau}x\coloneqq \sum_{j=1}^n
 		f_j(x)\tau_j \in\mathcal{X}.
 	\end{align*}
 	\item The \textbf{analysis operator} is defined as 
 	\begin{align*}
 		\theta_f:\mathcal{X}\ni x \mapsto \theta_fx\coloneqq (f_j(x))_{j=1}^n \in \mathbb{K}^n.
 	\end{align*}
 	\item The \textbf{synthesis operator} is defined as 
 		\begin{align*}
 		\theta_\tau: \mathbb{K}^n \ni (a_j)_{j=1}^n \mapsto \theta_\tau (a_j)_{j=1}^n \coloneqq \sum_{j=1}^n
 		a_j\tau_j \in\mathcal{X}.
 	\end{align*}
 \end{enumerate}
 \end{definition}
Following proposition follows easily.
\begin{proposition}
Given a collection $\{\tau_j\}_{j=1}^n$ in $\mathcal{X}$  and 	  a collection $\{f_j\}_{j=1}^n$  in  $\mathcal{X}^*$, the frame operator factors as $	S_{f, \tau}=\theta_\tau\theta_f$.	
\end{proposition}

\begin{definition}\cite{FREEMANODELL, THOMAS}
Let $\{\tau_j\}_{j=1}^n$ be a collection in  $\mathcal{X}$ and 	$\{f_j\}_{j=1}^n$ be a collection in  $\mathcal{X}^*.$ The pair $ (\{f_j \}_{j=1}^n, \{\tau_j \}_{j=1}^n) $ is said to be an \textbf{approximate  Schauder frame (ASF)} for $\mathcal{X}$ if the map 
\begin{align*}
	S_{f, \tau}:\mathcal{X}\ni x \mapsto S_{f, \tau}x\coloneqq \sum_{j=1}^n
	f_j(x)\tau_j \in\mathcal{X}.
\end{align*}
is invertible. If $S_{f, \tau}=\lambda I_\mathcal{X}$, for some non zero scalar $\lambda$, then $ (\{f_j \}_{j=1}^n, \{\tau_j \}_{j=1}^n) $ is called as a \textbf{tight ASF} for $\mathcal{X}$. 
\end{definition} 
Following theorem says that we can recover the trace of  frame operator using ASFs.
\begin{theorem}\label{DFBS}
Given a collection $\{\tau_j\}_{j=1}^n$ in $\mathcal{X}$  and 	  a collection $\{f_j\}_{j=1}^n$  in  $\mathcal{X}^*$, we have 
\begin{align*}
&\operatorname{Tra}(S_{f,\tau})=\sum_{j=1}^nf_j(\tau_j), \\
& \operatorname{Tra}(S^2_{f,\tau})=\sum_{j=1}^n\sum_{k=1}^nf_j(\tau_k)f_k(\tau_j).
\end{align*}	
\end{theorem}
\begin{proof}
From the expression of $S_{f,\tau}$ and from the  definition of trace of operator in Banach space (see \cite{JOHNSONSZANKOWSKI}) we get

\begin{align*}
\operatorname{Tra}(S_{f,\tau})=\sum_{j=1}^nf_j(\tau_j).
\end{align*}
Now 
\begin{align*}
S_{f,\tau}^2x=\sum_{j=1}^n
f_j(x)(S_{f,\tau}\tau_j), \quad \forall x \in \mathcal{X}.
\end{align*}
 Again using the definition of trace we get 
  \begin{align*}
  \operatorname{Tra}(S^2_{f,\tau})&=\sum_{j=1}^nf_j(S_{f,\tau}\tau_j)=\sum_{j=1}^nf_j\left(\sum_{k=1}^nf_k(\tau_j)\tau_k\right)\\
  &=\sum_{j=1}^n\sum_{k=1}^nf_j(\tau_k)f_k(\tau_j).
  \end{align*}
\end{proof}
Now we can derive the first important result of the paper.
  \begin{theorem}\label{FIRSTWELCHBANACH}
  \textbf{(First order  Welch bound for Banach spaces)}	Let $\{\tau_j\}_{j=1}^n$ be a collection in a finite dimensional Banach space $\mathcal{X}$ of dimension $d$  and 	$\{f_j\}_{j=1}^n$ be a collection in  $\mathcal{X}^*.$ Let $n\geq d$. If the  operator 
  $	S_{f, \tau}:\mathcal{X}\ni x \mapsto S_{f, \tau}x\coloneqq \sum_{j=1}^n
  	f_j(x)\tau_j \in
  	\mathcal{X}$ is diagonalizable and its eigenvalues are all non negative, then 
  	\begin{align}\label{GENERALIZEDWELCH}
 \sum_{1\leq j, k \leq n}f_j(\tau_k)f_k(\tau_j)= \sum_{j=1}^n\sum_{k=1}^nf_j(\tau_k)f_k(\tau_j)\geq \frac{1}{d}	\left(\sum_{j=1}^n
  f_j(\tau_j)\right)^2=\frac{1}{\text{dim}(\mathcal{X})}	\left(\sum_{j=1}^n
  f_j(\tau_j)\right)^2.
  	\end{align}
  	and 
  	
  	\begin{align*}
  \max _{1\leq j,k \leq n, j\neq k}|f_j(\tau_k)f_k(\tau_j)|&\geq \frac{\frac{1}{d}	\left(\sum_{j=1}^n
  	f_j(\tau_j)\right)^2-	\sum_{j=1}^n
  	|f_j(\tau_j)|^2}{n^2-n}\\
  &=\frac{\frac{1}{\text{dim}(\mathcal{X})}	\left(\sum_{j=1}^n
  	f_j(\tau_j)\right)^2-	\sum_{j=1}^n
  	|f_j(\tau_j)|^2}{n^2-n},
  	\end{align*}
  	\begin{align*}
  	\max _{1\leq j,k \leq n, j\neq k}|f_j(\tau_k)|&\geq \sqrt{\frac{\frac{1}{d}	\left(\sum_{j=1}^n
  			f_j(\tau_j)\right)^2-	\sum_{j=1}^n
  			|f_j(\tau_j)|^2}{n^2-n}}\\
  		&=\sqrt{\frac{\frac{1}{\text{dim}(\mathcal{X})}	\left(\sum_{j=1}^n
  				f_j(\tau_j)\right)^2-	\sum_{j=1}^n
  				|f_j(\tau_j)|^2}{n^2-n}}.
  	\end{align*}
  	Further, equality holds in Inequality (\ref{GENERALIZEDWELCH}) if and only if $ (\{f_j \}_{j=1}^n, \{\tau_j \}_{j=1}^n) $ is  tight  ASF for $\mathcal{X}$. 
  	In particular, if $f_j(\tau_j)=1$ for all $1\leq j \leq n$, then 
  \begin{align*}
  \max _{1\leq j,k \leq n, j\neq k}|f_j(\tau_k)f_k(\tau_j)|\geq \frac{n-d}{d(n-1)}=\frac{n-\text{dim}(\mathcal{X})}{\text{dim}(\mathcal{X})(n-1)}
  \end{align*}
  and we have \textbf{first order (discrete) Welch bound for Banach spaces}
  \begin{align*}
	 \max _{1\leq j,k \leq n, j\neq k}|f_j(\tau_k)|\geq \sqrt{\frac{n-d}{d(n-1)}}= \sqrt{\frac{n-\text{dim}(\mathcal{X})}{\text{dim}(\mathcal{X})(n-1)}}.
  \end{align*}
  \end{theorem}
  \begin{proof}
  	Let $\lambda_1, \dots, \lambda_d$ be eigenvalues of $	S_{f, \tau}$. Then $\lambda_1, \dots, \lambda_d\geq0$ and using the diagonalizability of $	S_{f, \tau}$  we get 
  	\begin{align}\label{INTERMEDIATE}
  \left(\sum_{j=1}^n
  f_j(\tau_j)\right)^2&=	(\operatorname{Tra}(S_{f,\tau}))^2=\left(\sum_{k=1}^d
  \lambda_k\right)^2\leq d \sum_{k=1}^d
  \lambda_k^2 \nonumber\\
  &=d\operatorname{Tra}(S^2_{f,\tau})=d\sum_{j=1}^n\sum_{k=1}^nf_j(\tau_k)f_k(\tau_j).
  	\end{align}
 For the second inequality,  
   \begin{align*}
  \frac{1}{d} \left(\sum_{j=1}^n
  f_j(\tau_j)\right)^2&\leq \sum_{j=1}^n\sum_{k=1}^nf_j(\tau_k)f_k(\tau_j)=\sum_{j=1}^nf_j(\tau_j)^2+\sum_{j,k=1, j\neq k}^nf_j(\tau_k)f_k(\tau_j)\\
  &\leq \sum_{j=1}^n|f_j(\tau_j)|^2+\sum_{j,k=1, j\neq k}^n|f_j(\tau_k)f_k(\tau_j)|\\
  &\leq \sum_{j=1}^n|f_j(\tau_j)|^2+ (n^2-n)\max _{1\leq j,k \leq n, j\neq k}|f_j(\tau_k)f_k(\tau_j)|
  \end{align*}	
  which gives 
  \begin{align}\label{PROOFFIRST}
  \max _{1\leq j,k \leq n, j\neq k}|f_j(\tau_k)f_k(\tau_j)|\geq \frac{\frac{1}{d}	\left(\sum_{j=1}^n
  	f_j(\tau_j)\right)^2-	\sum_{j=1}^n
  	|f_j(\tau_j)|^2}{n^2-n}.
  \end{align}
  Now let $1\leq j,k \leq n, j\neq k$ be fixed. Then 
  
  \begin{align*}
  |f_j(\tau_k)f_k(\tau_j)|&\leq \max _{1\leq j,k \leq n, j\neq k}|f_j(\tau_k)|\max _{1\leq j,k \leq n, j\neq k}|f_k(\tau_j)|\\
  &=\max _{1\leq j,k \leq n, j\neq k}|f_j(\tau_k)|\max _{1\leq j,k \leq n, j\neq k}|f_j(\tau_k)|=\left(\max _{1\leq j,k \leq n, j\neq k}|f_j(\tau_k)|\right)^2.
  \end{align*}
  Therefore 
  \begin{align*}
  \max _{1\leq j,k \leq n, j\neq k}|f_j(\tau_k)f_k(\tau_j)|\leq \left(\max _{1\leq j,k \leq n, j\neq k}|f_j(\tau_k)|\right)^2
  \end{align*}
  Using Inequality (\ref{PROOFFIRST}) we now get 
  	\begin{align}\label{PROOFSECOND}
  \max _{1\leq j,k \leq n, j\neq k}|f_j(\tau_k)|\geq \sqrt{\frac{\frac{1}{d}	\left(\sum_{j=1}^n
  		f_j(\tau_j)\right)^2-	\sum_{j=1}^n
  		|f_j(\tau_j)|^2}{n^2-n}}.
  \end{align}
  Whenever $f_j(\tau_j)=1$ for all $1\leq j \leq n$, Inequality  (\ref{PROOFFIRST}) gives 
  \begin{align*}
\max _{1\leq j,k \leq n, j\neq k}|f_j(\tau_k)f_k(\tau_j)|\geq \frac{\frac{n^2}{d}-n}{n^2-n}=\frac{n-d}{d(n-1)}.
  \end{align*}
   and Inequality (\ref{PROOFSECOND}) gives 
   \begin{align*}
   \max _{1\leq j,k \leq n, j\neq k}|f_j(\tau_k)|\geq \sqrt{\frac{n-d}{d(n-1)}}.
   \end{align*}
   Equality   holds in Inequality  (\ref{INTERMEDIATE}) if and only if 
   \begin{align*}
   \left(\sum_{k=1}^d
   \lambda_k\right)^2= \left(\sum_{k=1}^d1\right) \left(\sum_{k=1}^d
   \lambda_k^2\right)
   \end{align*}
   if and only if 
   \begin{align*}
   \lambda_k=a, \text{ for some } a>0, \forall 1\leq k \leq n
   \end{align*}
   if and only if $S_{f,\tau}$ is a tight ASF for $\mathcal{X}$. 
  \end{proof}
  We next derive higher order Welch bounds for Banach spaces. First we need a standard result.

  \begin{theorem}\cite{COMON, BOCCI}\label{SYMMETRICTENSORDIMENSION}
  	If $\mathcal{V}$ is a vector space of dimension $d$ and $\text{Sym}^m(\mathcal{V})$ denotes the vector space of symmetric m-tensors, then 
  	\begin{align*}
  		\text{dim}(\text{Sym}^m(\mathcal{V}))={d+m-1 \choose m}, \quad \forall m \in \mathbb{N}.
  	\end{align*}
  \end{theorem}
  
  \begin{theorem}\label{WELCHBANACH}
(\textbf{Welch bounds for Banach spaces})  	Let $\{\tau_j\}_{j=1}^n$ be a collection in a finite dimensional Banach space $\mathcal{X}$ of dimension $d$  and 	$\{f_j\}_{j=1}^n$ be a collection in  $\mathcal{X}^*.$ Let $n\geq d$ and $m\in \mathbb{N}$. If the  operator 
  $	S_{f, \tau}:\text{Sym}^m(\mathcal{X})\ni x \mapsto S_{f, \tau}x\coloneqq \sum_{j=1}^n
  f_j^{\otimes m}(x)\tau_j ^{\otimes m}\in
 \text{Sym}^m(\mathcal{X})$ is diagonalizable and its eigenvalues are all non negative, then 
  \begin{align}\label{EQUALITYIFFFRAMETIGHT}
 \sum_{1\leq j,k \leq n}f_j(\tau_k)^mf_k(\tau_j)^m &=\sum_{j=1}^n\sum_{k=1}^nf_j(\tau_k)^mf_k(\tau_j)^m
 \geq \frac{1}{{d+m-1 \choose m}}	\left(\sum_{j=1}^n
  f_j(\tau_j)^m\right)^2\nonumber \\
  &=\frac{1}{\text{dim}(\text{Sym}^m(\mathcal{X}))}	\left(\sum_{j=1}^n
  f_j(\tau_j)^m\right)^2.
  \end{align}
  and 
  \begin{align*}
  \max _{1\leq j,k \leq n, j\neq k}|f_j(\tau_k)f_k(\tau_j)|^m&\geq \frac{\frac{1}{{d+m-1 \choose m}}	\left(\sum_{j=1}^n
  	f_j(\tau_j)^m\right)^2-	\sum_{j=1}^n
  	|f_j(\tau_j)|^{2m}}{n^2-n}\\
  &=\frac{\frac{1}{\text{dim}(\text{Sym}^m(\mathcal{X}))}	\left(\sum_{j=1}^n
  	f_j(\tau_j)^m\right)^2-	\sum_{j=1}^n
  	|f_j(\tau_j)|^{2m}}{n^2-n},
  \end{align*}
  \begin{align*}
  \max _{1\leq j,k \leq n, j\neq k}|f_j(\tau_k)|^m&\geq \sqrt{\frac{\frac{1}{{d+m-1 \choose m}}	\left(\sum_{j=1}^n
  		f_j(\tau_j)^m\right)^2-	\sum_{j=1}^n
  		|f_j(\tau_j)|^{2m}}{n^2-n}}\\
  	&=\sqrt{\frac{\frac{1}{\text{dim}(\text{Sym}^m(\mathcal{X}))}	\left(\sum_{j=1}^n
  			f_j(\tau_j)^m\right)^2-	\sum_{j=1}^n
  			|f_j(\tau_j)|^{2m}}{n^2-n}}.
  \end{align*}
  Further, equality holds in Inequality (\ref{EQUALITYIFFFRAMETIGHT}) if and only if $ (\{f_j \}_{j=1}^n, \{\tau_j \}_{j=1}^n) $ is a  tight  ASF for $\text{Sym}^m(\mathcal{X})$. 
  In particular, if $f_j(\tau_j)=1$ for all $1\leq j \leq n$, then 
  \begin{align*}
  \max _{1\leq j,k \leq n, j\neq k}|f_j(\tau_k)f_k(\tau_j)|^m&\geq \frac{n-{d+m-1 \choose m}}{{d+m-1 \choose m}(n-1)}=\frac{n-\text{dim}(\text{Sym}^m(\mathcal{X}))}{\text{dim}(\text{Sym}^m(\mathcal{X}))(n-1)}
  \\
  &=\frac{1}{n-1}\left[\frac{n}{{d+m-1\choose m}}-1\right]
  \end{align*}
  and we have \textbf{ higher order (discrete) Welch bounds for Banach spaces}
  \begin{align*}
 \max _{1\leq j,k \leq n, j\neq k}|f_j(\tau_k)|^m&\geq\sqrt{\frac{n-{d+m-1 \choose m}}{{d+m-1 \choose m}(n-1)}}= \sqrt{\frac{n-\text{dim}(\text{Sym}^m(\mathcal{X}))}{\text{dim}(\text{Sym}^m(\mathcal{X}))(n-1)}}\\
 &=\sqrt{\frac{1}{n-1}\left[\frac{n}{{d+m-1\choose m}}-1\right]}.
  \end{align*}	
  \end{theorem}
  \begin{proof}
  	We will do  the proof of Theorem \ref{FIRSTWELCHBANACH} for the space $\text{Sym}^m(\mathcal{X})$ (we refer \cite{RYAN} for the tensor product of Banach spaces). Let $\lambda_1, \dots, \lambda_{\text{dim}(\text{Sym}^m(\mathcal{X}))}$ be eigenvalues of $	S_{f, \tau}$. Then 
  		\begin{align*}
  \left(\sum_{j=1}^n
  f_j(\tau_j)^m\right)^2	&=\left(\sum_{j=1}^n
  	f_j^{\otimes m}(\tau_j^{\otimes m})\right)^2=	(\operatorname{Tra}(S_{f,\tau}))^2=\left(\sum_{l=1}^{\text{dim}(\text{Sym}^m(\mathcal{X}))}
  	\lambda_l\right)^2\\
  	&\leq 
  	\text{dim}(\text{Sym}^m(\mathcal{X})) \sum_{l=1}^{\text{dim}(\text{Sym}^m(\mathcal{X}))}
  	\lambda_l^2
  	={d+m-1 \choose m}\operatorname{Tra}(S^2_{f,\tau})\\
  	&={d+m-1 \choose m}\sum_{j=1}^n\sum_{l=1}^nf_j^{\otimes m}(\tau_l^{\otimes m})f_l^{\otimes m}(\tau_j^{\otimes m})
  	\\
  	&={d+m-1 \choose m}\sum_{j=1}^n\sum_{k=1}^nf_j(\tau_k)^mf_k(\tau_j)^m
  	\end{align*}
  and 
  
  \begin{align*}
  	\frac{1}{{d+m-1 \choose m}}  \left(\sum_{j=1}^n
  	f_j(\tau_j)^m\right)^2&=\sum_{1\leq j,k \leq n}f_j(\tau_k)^mf_k(\tau_j)^m=\sum_{1\leq j,k \leq n, j\neq k}f_j(\tau_k)^mf_k(\tau_j)^m+\sum_{j=1}^nf_j(\tau_j)^{2m}\\
  	&\leq \sum_{1\leq j,k \leq n, j\neq k}|f_j(\tau_k)f_k(\tau_j)|^m+\sum_{j=1}^n|f_j(\tau_j)|^{2m}\\
  	&\leq (n^2-n) \max _{1\leq j,k \leq n, j\neq k}|f_j(\tau_k)f_k(\tau_j)|^m+\sum_{j=1}^n|f_j(\tau_j)|^{2m}.
  \end{align*}
  Other parts are similar to the corresponding part in the proof of Theorem \ref{FIRSTWELCHBANACH}.
  \end{proof}
\begin{corollary}
	Theorem  \ref{WELCHTHEOREM} is a corollary of Theorem 	\ref{WELCHBANACH}.
\end{corollary}
\begin{proof}
Let $\{\tau_j\}_{j=1}^n$ be a finite collection in a finite dimensional Hilbert space $\mathcal{H}$ of dimension $d$. Define $f_j: \mathcal{H}\ni h \mapsto \langle h, \tau_j \rangle \in \mathbb{K}$,  $\forall 1 \leq j \leq n$.  Let $m \in \mathbb{N}$. 	 Then the operator 
\begin{align*}
	S_\tau :\text{Sym}^m(\mathcal{H})\ni h \mapsto \sum_{j=1}^{n}f_j^{\otimes m}(h)\tau_j^{\otimes m} =\sum_{j=1}^{n}\langle h, \tau_j^{\otimes m} \rangle \tau_j^{\otimes m}\in \text{Sym}^m(\mathcal{H})
\end{align*}
is positive and Theorem \ref{WELCHBANACH} applies.
\end{proof}

  Theorem \ref{WELCHBANACH} gives Welch bound for all natural numbers. One can now ask whether we can replace naturals by positive reals. Following results show that we can do this.   For normalized tight frames for Hilbert spaces, these results are  derived in \cite{HAIKINZAMIRGAVISH} and   \cite{EHLEROKOUDJOU}.
  \begin{theorem}\label{DISCRETEARBITRARY1}
 	Let $\{\tau_j\}_{j=1}^n$ be a collection in a finite dimensional Banach space $\mathcal{X}$ of dimension $d$  and 	$\{f_j\}_{j=1}^n$ be a collection in  $\mathcal{X}^*.$ Let $n\geq d$. If the  operator 
 $	S_{f, \tau}:\mathcal{X}\ni x \mapsto S_{f, \tau}x\coloneqq \sum_{j=1}^n
 f_j(x)\tau_j \in
 \mathcal{X}$ is diagonalizable and its eigenvalues are all non negative, then  
 
 \begin{align*}
 \operatorname{Tra}(S^r_{f,\tau})\geq \frac{\left(\sum_{j=1}^n
 	f_j(\tau_j)\right)^r}{d^{r-1}},\quad \forall r \in [1, \infty)
 \end{align*}
and 
 \begin{align*}
	\operatorname{Tra}(S^r_{f,\tau})\leq \frac{\left(\sum_{j=1}^n
		f_j(\tau_j)\right)^r}{d^{r-1}},\quad \forall r \in (0,1).
\end{align*}
 In particular, if $f_j(\tau_j)=1$ for all $1\leq j \leq n$, then 
 \begin{align*}
\operatorname{Tra}(S^r_{f,\tau})\geq \frac{n^r}{d^{r-1}}, \quad \forall r \in [1, \infty)
\end{align*} 	
and 
 \begin{align*}
	\operatorname{Tra}(S^r_{f,\tau})\leq \frac{n^r}{d^{r-1}},\quad \forall r \in (0,1).
\end{align*} 	
 \end{theorem}
\begin{proof}
	Let $\lambda_1, \dots, \lambda_d$ be eigenvalues of $	S_{f, \tau}$. Let $r \in [1, \infty)$.	Using Jensen's inequality, we have 
	\begin{align*}
	\left(\frac{\sum_{j=1}^n
		f_j(\tau_j)}{d}\right)^r=\left(\frac{\operatorname{Tra}(S_{f,\tau})}{d}\right)^r=\left(\frac{\sum_{k=1}^{d}\lambda_k}{d}\right)^r\leq \frac{\sum_{k=1}^{d}\lambda_k^r}{d}=\frac{1}{d}\operatorname{Tra}(S^r_{f,\tau})
	\end{align*}
	which gives the first part. Second part again follows from Jensen's inequality.
\end{proof}
\begin{theorem}\label{DISCRETEARBITRARY2}
		Let $\{\tau_j\}_{j=1}^n$ be a collection in a finite dimensional Banach space $\mathcal{X}$ of dimension $d$  and 	$\{f_j\}_{j=1}^n$ be a collection in  $\mathcal{X}^*.$  Let $2<p<\infty$. If the  operator 
	$	S_{f, \tau}:\mathcal{X}\ni x \mapsto S_{f, \tau}x\coloneqq \sum_{j=1}^n
	f_j(x)\tau_j \in
	\mathcal{X}$ is diagonalizable and its eigenvalues are all non negative, then  
	\begin{align*}
		 \sum_{1\leq j,k \leq n}|f_j(\tau_k)f_k(\tau_j)|^\frac{p}{2} \geq n(n-1)\left(\frac{n-d}{d(n-1)}\right)^\frac{p}{2}+n.
	\end{align*}
\end{theorem}
\begin{proof}
	Define $r\coloneqq 2p/(p-2)$ and $q$ be the conjugate index of $p/2$. Then $q=r/2$. Using Theorem \ref{FIRSTWELCHBANACH} and Holder's inequality, we  have 
\begin{align*}
	\frac{n^2}{d}-n&\leq  \sum_{1\leq j,k \leq n, j\neq k}|f_j(\tau_k)f_k(\tau_j)| \leq \left(\sum_{1\leq j,k \leq n, j\neq k}|f_j(\tau_k)f_k(\tau_j)|^\frac{p}{2}\right)^\frac{2}{p}\left( \sum_{1\leq j,k \leq n, j\neq k} 1\right)^\frac{1}{q}\\
	&=\left(\sum_{1\leq j,k \leq n, j\neq k}|f_j(\tau_k)f_k(\tau_j)|^\frac{p}{2}\right)^\frac{2}{p}(n^2-n)^\frac{1}{q}\\	
	&=\left(\sum_{1\leq j,k \leq n, j\neq k}|f_j(\tau_k)f_k(\tau_j)|^\frac{p}{2}\right)^\frac{2}{p}(n^2-n)^\frac{2}{r}\\	
	&=\left(\sum_{1\leq j,k \leq n, j\neq k}|f_j(\tau_k)f_k(\tau_j)|^\frac{p}{2}\right)^\frac{2}{p}(n^2-n)^\frac{p-2}{p}	
\end{align*}
which gives 
\begin{align*}
	\left(\frac{n^2}{d}-n\right)^\frac{p}{2}\leq \left(\sum_{1\leq j,k \leq n, j\neq k}|f_j(\tau_k)f_k(\tau_j)|^\frac{p}{2}\right)(n^2-n)^{\frac{p}{2}-1}.
\end{align*}
Therefore 

\begin{align*}
	n(n-1)\left(\frac{n-d}{d(n-1)}\right)^\frac{p}{2}+n&=\frac{1}{(n^2-n)^{\frac{p}{2}-1}}	\left(\frac{n^2}{d}-n\right)^\frac{p}{2}+n\\
	&\leq \sum_{1\leq j,k \leq n, j\neq k}|f_j(\tau_k)f_k(\tau_j)|^\frac{p}{2}+\sum_{j=1}^{n}|f_j(\tau_j)f_j(\tau_j)|^\frac{p}{2}\\
&=\sum_{1\leq j,k \leq n}|f_j(\tau_k)f_k(\tau_j)|^\frac{p}{2}.
\end{align*}	
\end{proof}
Some of the proofs of Theorem \ref{WELCHBANACH} (for instance see \cite{ROSENFELD}) use the idea of Gram matrix and Frobenius norm/Hilbert-Schmidt norm. We now give Welch bound for Banach spaces using matrices. First we need a definition.
\begin{definition}\label{GRAMBANACH}
Let $\{\tau_j\}_{j=1}^n$ be a collection in a  Banach space $\mathcal{X}$  and 	$\{f_j\}_{j=1}^n$ be a collection in  $\mathcal{X}^*.$	We define the \textbf{Gram matrix} $G_{f,\tau}$ of  $ (\{f_j \}_{j=1}^n, \{\tau_j \}_{j=1}^n) $ as 
\begin{align*}
G_{f,\tau}\coloneqq [f_j(\tau_k)]_{1\leq j,k \leq n}=\begin{pmatrix}
f_1(\tau_1) & f_1(\tau_2) & \cdots & f_1(\tau_n)\\
f_2(\tau_1) & f_2(\tau_2) & \cdots & f_2(\tau_n)\\
\vdots & \vdots & & \vdots \\
f_n(\tau_1) & f_n(\tau_2) & \cdots & f_n(\tau_n)\\
\end{pmatrix}_{n\times n} 
\in \mathbb{M}_n(\mathbb{K}).
\end{align*}
\end{definition}
In terms of analysis and synthesis operators, $G_{f,\tau}=\theta_f\theta_\tau$. We observe that Definition  \ref{GRAMBANACH}  reduces to the definition of Gram matrix in Hilbert spaces. Indeed, let $\{\tau_j\}_{j=1}^n$ be a collection in a  Hilbert space $\mathcal{H}$. Now define $f_j(h)\coloneqq \langle h, \tau_j\rangle $ for all $h \in \mathcal{H}$, for all $1\leq j \leq n$.
\begin{theorem}\label{GRAMFIRST}
Let $\{\tau_j\}_{j=1}^n$ be a collection in a  Banach space $\mathcal{X}$  and 	$\{f_j\}_{j=1}^n$ be a collection in  $\mathcal{X}^*.$  If the Gram matrix $G_{f,\tau}=[f_j(\tau_k)]_{1\leq j,k \leq n}$ is diagonalizable and its eigenvalues are all non negative, then 
\begin{align*}
\sum_{1\leq j,k \leq n }f_j(\tau_k)f_k(\tau_j)=\sum_{j=1}^n\sum_{k=1}^nf_j(\tau_k)f_k(\tau_j)\geq \frac{1}{\text{rank}(G_{f,\tau})}	\left(\sum_{j=1}^n
f_j(\tau_j)\right)^2.
\end{align*}
and 
\begin{align*}
\max _{1\leq j,k \leq n, j\neq k}|f_j(\tau_k)f_k(\tau_j)|\geq \frac{\frac{1}{\text{rank}(G_{f,\tau})}	\left(\sum_{j=1}^n
	f_j(\tau_j)\right)^2-	\sum_{j=1}^n
	|f_j(\tau_j)|^2}{n^2-n},
\end{align*}
\begin{align*}
\max _{1\leq j,k \leq n, j\neq k}|f_j(\tau_k)|\geq \sqrt{\frac{\frac{1}{\text{rank}(G_{f,\tau})}	\left(\sum_{j=1}^n
		f_j(\tau_j)\right)^2-	\sum_{j=1}^n
		|f_j(\tau_j)|^2}{n^2-n}}.
\end{align*}
In particular, if $f_j(\tau_j)=1$ for all $1\leq j \leq n$, then 
\begin{align*}
\max _{1\leq j,k \leq n, j\neq k}|f_j(\tau_k)f_k(\tau_j)|\geq \frac{n-\text{rank}(G_{f,\tau})}{\text{rank}(G_{f,\tau})(n-1)}
\end{align*}
and 
\begin{align*}
\max _{1\leq j,k \leq n, j\neq k}|f_j(\tau_k)|\geq \sqrt{\frac{n-\text{rank}(G_{f,\tau})}{\text{rank}(G_{f,\tau})(n-1)}}.
\end{align*}	
\end{theorem}
\begin{proof}
Let $\lambda_1, \dots, \lambda_{(\text{rank}G_{f,\tau})}$ be the  eigenvalues	of $G_{f,\tau}$. Then 

\begin{align*}
  \left(\sum_{j=1}^n
f_j(\tau_j)\right)^2&=	(\operatorname{Trace}(G_{f,\tau}))^2=\left(\sum_{k=1}^{\text{rank}(G_{f,\tau})}
\lambda_k\right)^2\leq {\text{rank}(G_{f,\tau})} \sum_{k=1}^{\text{rank}(G_{f,\tau})}
\lambda_k^2\\
&=\text{rank}(G_{f,\tau})\operatorname{Trace}(G^2_{f,\tau})=\text{rank}(G_{f,\tau})\sum_{j=1}^n\sum_{k=1}^nf_j(\tau_k)f_k(\tau_j)\\
&\leq \text{rank}(G_{f,\tau})\left(\sum_{1\leq j,k \leq n, j\neq k}|f_j(\tau_k)f_k(\tau_j)|+\sum_{j=1}^{n}|f_j(\tau_j)|^2\right)\\
&\leq \text{rank}(G_{f,\tau})\left((n^2-n)\max_{1\leq j,k \leq n, j\neq k}|f_j(\tau_k)f_k(\tau_j)|+\sum_{j=1}^{n}|f_j(\tau_j)|^2\right).
\end{align*}
\end{proof}

In the following result, given a matrix $G$,  $G^{\circ^m}$ denotes the Hadamard/Schur/pointwise product of $G$ with itself, $m$ times \cite{HORNHADAMARD}.
\begin{theorem}
Let $\{\tau_j\}_{j=1}^n$ be a collection in a  Banach space $\mathcal{X}$  and 	$\{f_j\}_{j=1}^n$ be a collection in  $\mathcal{X}^*.$ Let $m\in \mathbb{N}$. If the Hadamard product  $G_{f,\tau}^{\circ^m}$ is diagonalizable and its eigenvalues are all non negative, then 

\begin{align*}
	\sum_{1\leq j,k \leq n }f_j(\tau_k)^mf_k(\tau_j)^m=\sum_{j=1}^n\sum_{k=1}^nf_j(\tau_k)^mf_k(\tau_j)^m\geq \frac{1}{\text{rank}(G_{f,\tau}^{\circ^m})}	\left(\sum_{j=1}^n
	f_j(\tau_j)^m\right)^2.
\end{align*}
and 
\begin{align*}
	\max _{1\leq j,k \leq n, j\neq k}|f_j(\tau_k)f_k(\tau_j)|^m\geq \frac{\frac{1}{\text{rank}(G_{f,\tau}^{\circ^m})}	\left(\sum_{j=1}^n
		f_j(\tau_j)^m\right)^2-	\sum_{j=1}^n
		|f_j(\tau_j)|^{2m}}{n^2-n},
\end{align*}
\begin{align*}
	\max _{1\leq j,k \leq n, j\neq k}|f_j(\tau_k)|^m\geq \sqrt{\frac{\frac{1}{\text{rank}(G_{f,\tau}^{\circ^m})}	\left(\sum_{j=1}^n
			f_j(\tau_j)^m\right)^2-	\sum_{j=1}^n
			|f_j(\tau_j)|^{2m}}{n^2-n}}.
\end{align*}
In particular, if $f_j(\tau_j)=1$ for all $1\leq j \leq n$, then 
\begin{align*}
	\max _{1\leq j,k \leq n, j\neq k}|f_j(\tau_k)f_k(\tau_j)|^m\geq \frac{n-\text{rank}(G_{f,\tau}^{\circ^m})}{\text{rank}(G_{f,\tau}^{\circ^m})(n-1)}
\end{align*}
and 
\begin{align*}
	\max _{1\leq j,k \leq n, j\neq k}|f_j(\tau_k)|^m\geq \sqrt{\frac{n-\text{rank}(G_{f,\tau}^{\circ^m})}{\text{rank}(G_{f,\tau}^{\circ^m})(n-1)}}.
\end{align*}		
\end{theorem}
\begin{proof}
	We note that 
	\begin{align*}
	G_{f,\tau}^{\circ^m}=[(f_j(\tau_k))^m]_{1\leq j,k \leq n}=[(f_j^{\otimes m}(\tau^{\otimes m}_k))]_{1\leq j,k \leq n}.
	\end{align*}
Now the proof is similar to the proof of Theorem \ref{GRAMFIRST}.
\end{proof}

We next recall the definition of Gerzon's bound which allows us to recall  the bounds which are in the same way  to discrete Welch bounds in the Hilbert spaces. 
 \begin{definition}\cite{JASPERKINGMIXON}
	Given $d\in \mathbb{N}$, define \textbf{Gerzon's bound}
	\begin{align*}
		\mathcal{Z}(d, \mathbb{K})\coloneqq 
		\left\{ \begin{array}{cc} 
			d^2 & \quad \text{if} \quad \mathbb{K} =\mathbb{C}\\
			\frac{d(d+1)}{2} & \quad \text{if} \quad \mathbb{K} =\mathbb{R}.\\
		\end{array} \right.
	\end{align*}	
\end{definition}
\begin{theorem}\cite{JASPERKINGMIXON, XIACORRECTION, MUKKAVILLISABHAWALERKIPAAZHANG, SOLTANALIAN, BUKHCOX, CONWAYHARDINSLOANE, HAASHAMMENMIXON, RANKIN}  \label{LEVENSTEINBOUND}
	Define $m\coloneqq \operatorname{dim}_{\mathbb{R}}(\mathbb{K})/2$.	If	$\{\tau_j\}_{j=1}^n$  is any collection of  unit vectors in $\mathbb{K}^d$, then
	\begin{enumerate}[\upshape(i)]
		\item (\textbf{Bukh-Cox bound})
		\begin{align*}
			\max _{1\leq j,k \leq n, j\neq k}|\langle \tau_j, \tau_k\rangle |\geq \frac{\mathcal{Z}(n-d, \mathbb{K})}{n(1+m(n-d-1)\sqrt{m^{-1}+n-d})-\mathcal{Z}(n-d, \mathbb{K})}\quad \text{if} \quad n>d.
		\end{align*}
		\item (\textbf{Orthoplex/Rankin bound})	
		\begin{align*}
			\max _{1\leq j,k \leq n, j\neq k}|\langle \tau_j, \tau_k\rangle |\geq\frac{1}{\sqrt{d}} \quad \text{if} \quad n>\mathcal{Z}(d, \mathbb{K}).
		\end{align*}
		\item (\textbf{Levenstein bound})	
		\begin{align*}
			\max _{1\leq j,k \leq n, j\neq k}|\langle \tau_j, \tau_k\rangle |\geq \sqrt{\frac{n(m+1)-d(md+1)}{(n-d)(md+1)}} \quad \text{if} \quad n>\mathcal{Z}(d, \mathbb{K}).
		\end{align*}
		\item (\textbf{Exponential bound})
		\begin{align*}
			\max _{1\leq j,k \leq n, j\neq k}|\langle \tau_j, \tau_k\rangle |\geq 1-2n^{\frac{-1}{d-1}}.
		\end{align*}
	\end{enumerate}	
\end{theorem}
Theorem \ref{LEVENSTEINBOUND}  and Theorem \ref{FIRSTWELCHBANACH}  give  the following problem to us.
\begin{question}
	\textbf{ Whether there is a   version of Theorem \ref{LEVENSTEINBOUND} for Banach spaces?. In particular, does there exists a  (discrete) version of }
	\begin{enumerate}[\upshape(i)]
		\item  \textbf{ Bukh-Cox bound for Banach spaces?}
		\item  \textbf{ Orthoplex/Rankin bound for Banach spaces?}
		\item  \textbf{ Levenstein bound for Banach spaces?}
		\item  \textbf{ Exponential bound for Banach spaces?}
	\end{enumerate}		
\end{question}

\section{Applications of discrete Welch bounds for Banach spaces}\label{SECTIONTWO}
We begin by defining the RMS of vectors and functionals in Banach spaces.
 \begin{definition}\label{DISCRETERMS}
	Let $\{\tau_j\}_{j=1}^n$ be a collection in a finite dimensional Banach space $\mathcal{X}$ of dimension $d$  and 	$\{f_j\}_{j=1}^n$ be a collection in  $\mathcal{X}^*$ satisfying $f_j(\tau_j)=1$ for all $1\leq j \leq n$. Assume that  the  frame operator 
	$	S_{f, \tau}$  is diagonalizable and its eigenvalues are all non negative. We define the \textbf{root-mean-square (RMS) absolute cross relation} of    $ (\{f_j \}_{j=1}^n, \{\tau_j \}_{j=1}^n) $ as 
	\begin{align*}
		I_{\text{RMS}} (\{f_j \}_{j=1}^n, \{\tau_j \}_{j=1}^n)\coloneqq \left(\frac{1}{n(n-1)}\sum _{1\leq j,k \leq n, j\neq k} f_j(\tau_k)f_k(\tau_j)\right)^\frac{1}{2}.
	\end{align*}
\end{definition}
Theorem \ref{FIRSTWELCHBANACH}  gives the following result.
\begin{theorem}
	Let $\{\tau_j\}_{j=1}^n$ and $\{f_j\}_{j=1}^n$ be as in Definition \ref{DISCRETERMS}. Then
	\begin{align*}
	\max _{1\leq j,k \leq n, j\neq k}|f_j(\tau_k)| \geq 	I_{\text{RMS}} (\{f_j \}_{j=1}^n, \{\tau_j \}_{j=1}^n)\geq \left(\frac{n-d}{d(n-1)}\right)^\frac{1}{2}.
	\end{align*}
\end{theorem}
Here is another notion similar to that of frame potential in Hilbert spaces.
 \begin{definition}\label{DISCRETEPOTENTIAL}
	Let $\{\tau_j\}_{j=1}^n$ be a collection in a finite dimensional Banach space $\mathcal{X}$ of dimension $d$  and 	$\{f_j\}_{j=1}^n$ be a collection in  $\mathcal{X}^*$ satisfying $f_j(\tau_j)=1$ for all $1\leq j \leq n$. Assume that  the  frame operator 
	$	S_{f, \tau}$  is diagonalizable and its eigenvalues are all non negative. We define the \textbf{pseudo frame potential} of $ (\{f_j \}_{j=1}^n, \{\tau_j \}_{j=1}^n) $ as 
	\begin{align*}
		PFP (\{f_j \}_{j=1}^n, \{\tau_j \}_{j=1}^n)\coloneqq \sum_{j=1}^n\sum_{k=1}^n f_j(\tau_k)f_k(\tau_j).
	\end{align*}
\end{definition}
Note that we defined the notion pseudo frame potential and not frame potential. The reason is that frame potential for Banach spaces can not be defined in the way in Definition \ref{DISCRETEPOTENTIAL}. One has to go to the theory of p-summing operators (see \cite{DIESTELJARCHOWTONGE, NICOLE})  to define frame potential in Banach spaces, see \cite{FREEMANKORNELSON}. Theorem  \ref{FIRSTWELCHBANACH}  again gives the following result.
\begin{theorem}\label{PFPTHEOREM}
	Let $\{\tau_j\}_{j=1}^n$ and $\{f_j\}_{j=1}^n$ be as in Definition \ref{DISCRETEPOTENTIAL}. Then 	
	\begin{align}\label{PFPINENQUALITY}
n^2\max _{1\leq j,k \leq n}|f_j(\tau_k)|^2\geq 	PFP (\{f_j \}_{j=1}^n, \{\tau_j \}_{j=1}^n)	\geq \frac{n^2}{{d}}.	
	\end{align}
\end{theorem}
Theorem \ref{PFPTHEOREM}  gives the following problem.
 \begin{question}
	\textbf{Is  there  a characterization of $ (\{f_j \}_{j=1}^n, \{\tau_j \}_{j=1}^n) $ by using equality in Inequality (\ref{PFPINENQUALITY})?}
\end{question}
 We next introduce the notions of Grassmannian frames and equiangular frames for Banach spaes. First we need a definition.
  \begin{definition}
  	Let  $ (\{f_j \}_{j=1}^n, \{\tau_j \}_{j=1}^n) $ be an ASF  for $\mathcal{X}$ satisfying $f_j(\tau_j)=1$, $\forall 1\leq j \leq n$. 	Assume that  the  frame operator 
  	$	S_{f, \tau}$  is diagonalizable and its eigenvalues are all non negative. We define the  \textbf{frame correlation} of $ (\{f_j \}_{j=1}^n, \{\tau_j \}_{j=1}^n) $ as 
  	\begin{align*}
  		\mathcal{M}(\{f_j \}_{j=1}^n, \{\tau_j \}_{j=1}^n)\coloneqq \max _{1\leq j,k \leq n, j\neq k}|f_j(\tau_k)|.
  	\end{align*}
  \end{definition}
 
  \begin{definition}\label{GRASSMANNIANDEFINITION}
 	Let  $ (\{f_j \}_{j=1}^n, \{\tau_j \}_{j=1}^n) $ be an ASF  for $\mathcal{X}$ satisfying $\|f_j\|= 1, \|\tau_j\|=1, f_j(\tau_j)=1$, $\forall 1\leq j \leq n$. 	Assume that  the  frame operator 
 $	S_{f, \tau}$  is diagonalizable and its eigenvalues are all non negative.  ASF $ (\{f_j \}_{j=1}^n, \{\tau_j \}_{j=1}^n) $   is said to be a \textbf{Grassmannian frame} for  $\mathcal{X}$ if 
  	\begin{align*}
  		\mathcal{M}(\{f_j \}_{j=1}^n, \{\tau_j \}_{j=1}^n)=\inf\bigg\{&\mathcal{M}(\{g_j \}_{j=1}^n, \{\omega_j \}_{j=1}^n): (\{g_j \}_{j=1}^n, \{\omega_j \}_{j=1}^n) \text{ is an ASF for }\mathcal{X}  
  		  \text{satisfying } \\
  		  &\|g_j\|=1, \|\omega_j\|= 1, g_j(\omega_j)=1, \forall 1\leq j \leq n \text{ and the frame operator }  S_{g,\omega}\\
  		& \text{is diagonalizable and its eigenvalues are all non negative}\bigg\}.	
  	\end{align*}		
  \end{definition}
\begin{theorem}
Grassmannian frames exist in every dimension for every Banach space.
\end{theorem}
\begin{proof}
Our arguments are motivated from the arguments given in  \cite{BENEDETTONKOLESAR} for Hilbert spaces which mainly uses compactness and continuity. We give arguments only for real Banach spaces and complex case follows by considering real and imaginary parts (of the linear functionals).	Define
\begin{align*}
	&S_\mathcal{X}^n\coloneqq \{(x_1, \dots, x_n ): x_1, \dots, x_n \in \mathcal{X}, \|x_1\|=\cdots= \|x_n\|=1\},\\
	&S_\mathcal{X^*}^n\coloneqq \{(\phi_1, \dots, \phi_n ): \phi_1, \dots, \phi_n \in \mathcal{X}^*, \|\phi_1\|=\cdots= \|\phi_n\|=1\}\\
	&\mathcal{W} \coloneqq \{((x_1, \dots, x_n ), (\phi_1, \dots, \phi_n )): ((x_1, \dots, x_n ), (\phi_1, \dots ,\phi_n )) \in S_\mathcal{X}^n\times S_\mathcal{X^*}^n, \phi_j(x_j)=1, \forall 1\leq j \leq n,\\
	& \quad \quad  (\{\phi_j \}_{j=1}^n, \{x_j \}_{j=1}^n) \text{ is an ASF for }\mathcal{X}  \text{ and the frame operator }  S_{\phi,x}\\
	&  \quad \quad
	 \text{is diagonalizable and its eigenvalues are all non negative} \},\\
	&\Phi: \mathcal{W} \to [0,1], \quad \Phi((x_1, \dots, x_n ), (\phi_1, \dots, \phi_n ))\coloneqq 	\mathcal{M}(\{\phi_j \}_{j=1}^n, \{x_j \}_{j=1}^n).
\end{align*}
We re norm $\mathcal{X}^n\times \mathcal{X^*}^n$ by 
\begin{align*}
\|((x_1, \dots, x_n ), (\phi_1, \dots, \phi_n ))\|	\coloneqq \sum_{j=1}^{n}(\|x_j\|+\|\phi_j\|)
\end{align*}
and consider $\mathcal{W}$ in this norm.  Then $\mathcal{W}$ is compact. We show that $\Phi$ is continuous on $\mathcal{W}$ which proves the theorem. Let $ (\{\tau_j \}_{j=1}^n, \{f_j \}_{j=1}^n)$ $ \in \mathcal{W}$. Let $\varepsilon>0$ be given. Define 
\begin{align*}
	R\coloneqq 1+ \max_{1\leq j,k \leq n}\{\|\tau_j\|, \|f_k\|\}
\end{align*}
and let 
\begin{align*}
	0<\delta < \frac{\sqrt{1+\varepsilon}-1}{R}.
\end{align*}
Now for any given $ (\{\omega_j \}_{j=1}^n, \{g_j \}_{j=1}^n)$ $ \in \mathcal{W}$,  with 
\begin{align*}
	\|((\omega_1, \dots, \omega_n ), (g_1, \dots, g_n ))-((\tau_1, \dots, \tau_n ), (f_1, \dots, f_n ))\|<\delta,
\end{align*}
if we define $h_j\coloneqq g_j-f_j, \rho_j\coloneqq \omega_j-\tau_j, \forall 1\leq j \leq n,$ then $\|h_j\|< \delta, \|\rho_j\|<\delta$.   Then 
\begin{align*}
	&|\Phi((\omega_1, \dots, \omega_n ), (g_1, \dots, g_n ))-\Phi((\tau_1, \dots, \tau_n ), (f_1, \dots, f_n ))|=\left|\max _{1\leq j,k \leq n, j\neq k}|g_j(\omega_k)|-\max _{1\leq j,k \leq n, j\neq k}|f_j(\tau_k)|\right|\\	
		&=\max _{1\leq j,k \leq n, j\neq k}\big| |g_j(\omega_k)|-|f_j(\tau_k)|\big|\leq \max _{1\leq j,k \leq n, j\neq k}\big| g_j(\omega_k)-f_j(\tau_k)\big|\\
		&=\max _{1\leq j,k \leq n, j\neq k}\big| (f_j+h_j)(\tau_k+\rho_k)-f_j(\tau_k)\big|=\max _{1\leq j,k \leq n, j\neq k}\big| f_j(\tau_k)+f_j(\rho_k)+h_j(\tau_k)+h_j(\rho_k)-f_j(\tau_k)\big|\\
	&=\max _{1\leq j,k \leq n, j\neq k}|f_j(\rho_k)|+\max _{1\leq j,k \leq n, j\neq k}|h_j(\tau_k)|+\max _{1\leq j,k \leq n, j\neq k}|h_j(\rho_k)|\\
	&\leq \max _{1\leq j,k \leq n, j\neq k}\|f_j\|\|\rho_k\|+\max _{1\leq j,k \leq n, j\neq k}\|h_j\|\|\tau_k\|+\max _{1\leq j,k \leq n, j\neq k}\|h_j\|\|\rho_k\|\\
	&\leq R\cdot \delta+\delta \cdot R+\delta^2\leq R\cdot \delta+\delta \cdot R+R\delta^2<\varepsilon.
\end{align*}
\end{proof}
 
  \begin{definition}
	Let  $ (\{f_j \}_{j=1}^n, \{\tau_j \}_{j=1}^n) $ be an ASF  for $\mathcal{X}$ satisfying $f_j(\tau_j)=1$, $\forall 1\leq j \leq n$. 	Assume that  the  frame operator 
$	S_{f, \tau}$  is diagonalizable and its eigenvalues are all non negative.   ASF $ (\{f_j \}_{j=1}^n, \{\tau_j \}_{j=1}^n) $   is said to be 	\textbf{$\gamma$-equiangular} if there exists $\gamma\geq0$ such that
  	\begin{align*}
  		|f_j(\tau_k)|^2=\gamma, \quad \forall  1\leq j,k \leq n, j\neq k.
  	\end{align*}
  \end{definition}
 Here is the  Zauner's conjecture for Banach spaces (see \cite{APPLEBY123, APPLEBY, ZAUNER, SCOTTGRASSL, FUCHSHOANGSTACEY, RENESBLUMEKOHOUTSCOTTCAVES, APPLEBYSYMM, BENGTSSON, APPLEBYFLAMMIAMCCONNELLYARD, KOPPCON, GOURKALEV, BENGTSSONZYCZKOWSKI, PAWELRUDNICKIZYCZKOWSKI, BENGTSSON123, MAGSINO}  for a comprehensive look at Zauner's conjecture in Hilbert spaces and its connections with Hilbert's 12-problem and Stark conjecture).
  \begin{conjecture}
  	\textbf{(Zauner's conjecture for Banach spaces}) \textbf{For  every $d\in \mathbb{N}$, there exists a $\frac{1}{d+1}$-equiangular tight ASF  $ (\{f_j \}_{j=1}^{d^2}, \{\tau_j \}_{j=1}^{d^2}) $   for $\mathbb{C}^d$ (w.r.t. any norm) such that $\|f_j\|=\|\tau_j\|=|f_j(\tau_j)|=1, \forall1 \leq j \leq n$, i.e., there exist a tight ASF for  $\mathbb{C}^d$ such that $\|f_j\|=\|\tau_j\|=|f_j(\tau_j)|=1, \forall1 \leq j \leq n$ and 
  	\begin{align*}
  		|f_j(\tau_k)|^2=\frac{1}{d+1}, \quad \forall  1\leq j,k \leq d^2, j\neq k.	
  \end{align*}}
  \end{conjecture}
  
  Following theorem again follows from Theorem \ref{FIRSTWELCHBANACH}.
  \begin{theorem}\label{CNG2}
	Let  $ (\{f_j \}_{j=1}^n, \{\tau_j \}_{j=1}^n) $ be as in Definition \ref{GRASSMANNIANDEFINITION}.   Then 
  	\begin{align}\label{EQUIANGULARINEQUALITY}
  		\mathcal{M}(\{f_j \}_{j=1}^n, \{\tau_j \}_{j=1}^n )\geq \sqrt{\frac{n-d}{d(n-1)}} \eqqcolon\gamma.
  	\end{align}
  	If the ASF is $\gamma$-equiangular, then we have equality in Inequality (\ref{EQUIANGULARINEQUALITY}).
  \end{theorem}
\begin{question}
\begin{enumerate}[\upshape (i)]
	\item \textbf{	Whether the equality in Inequality (\ref{EQUIANGULARINEQUALITY}) implies that the ASF $ (\{f_j \}_{j=1}^n,$ $ \{\tau_j \}_{j=1}^n) $  is $\gamma$-equiangular?}
	\item \textbf{Whether the validity of Inequality (\ref{EQUIANGULARINEQUALITY}) implies there is a relation between the number of elements in the ASF and dimension of the space (like Theorem 2.3 in \cite{STROHMERHEATH})?}
\end{enumerate}
\end{question}

  \section{Continuous Welch bounds for Banach spaces}
 Unlike the discrete setting, we need some concepts to get continuous version of Theorem    \ref{WELCHBANACH} for Banach spaces. Following definition is motivated from the notion of continuous frames for Hilbert spaces \cite{ALIANTOINEGAZEAU, KAISER}, continuous framings for Banach spaces \cite{LILIHAN},  continuous Schauder frames for Banach spaces \cite{EISNERFREEMAN} and    approximate Schauder frames for Banach spaces \cite{FREEMANODELL, THOMAS}  (which is motivated from the notion of Schauder frame  \cite{CASAZZADILWORTH} which is motivated from the notion of framing \cite{CASAZZAHANLARSON}  which is motivated from the Han-Larson-Naimark dilation theorem \cite{CZAJA}  for Hilbert space frames \cite{HANLARSONMEMOIRS}).
  \begin{definition}
 	Let 	$(\Omega, \mu)$ be a measure space. Let    $\{\tau_\alpha\}_{\alpha\in \Omega}$ be a collection in a Banach   space $\mathcal{X}$ and     $\{f_\alpha\}_{\alpha\in \Omega}$ be a collection in  $\mathcal{X}^*$. The pair $(\{f_\alpha\}_{\alpha\in \Omega}, \{\tau_\alpha\}_{\alpha\in \Omega})$   is said to be a \textbf{continuous approximate Schauder frame}  for $\mathcal{X}$ if the following holds. 	
 	\begin{enumerate}[\upshape(i)]
 		\item For every $x\in \mathcal{X}$ and for every $\phi \in \mathcal{X}^*$, the map 
 		\begin{align*}
 			 \Omega \ni \alpha \mapsto f_\alpha(x)\phi(\tau_\alpha)\in \mathbb{K}
 		\end{align*}
 	is measurable and integrable. 
 		\item The \textbf{frame operator} 
 		\begin{align*}
 		S_{f,\tau}:\mathcal{X} \ni x  \mapsto S_{f,\tau}x\coloneqq \int_{\Omega}	f_\alpha (x)\tau_\alpha \, d \mu(\alpha) \in \mathcal{X}
 		\end{align*} 
 	is a well-defined invertible bounded linear operator,  where the 	integral is weak integral (Pettis integrals \cite{TALAGRAND}).
 	\end{enumerate}
If $S_{f, \tau}=\lambda I_\mathcal{X}$, for some non zero scalar $\lambda$, then $ (\{f_\alpha\}_{\alpha\in \Omega}, \{\tau_\alpha\}_{\alpha\in \Omega})$ is called as  a \textbf{tight continuous ASF} for $\mathcal{X}$.  If we do not demand the invertibility of $S_{f,\tau}$, then we say that $(\{f_\alpha\}_{\alpha\in \Omega}, \{\tau_\alpha\}_{\alpha\in \Omega})$ is a \textbf{continuous approximate Bessel family}  for $\mathcal{X}$.
  \end{definition}



  Our first  observation is  that there is a large supply   of continuous frames for finite dimensional Banach spaces.  Here is a  result of  existence of them for  Banach spaces. Our result is motivated from the work in \cite{RAHIMIDARABYDARVISHI}.
  \begin{theorem}\label{FBFRAMEEXIST}
  	Let $\mathcal{X}$ be a finite dimensional Banach space of dimension $d$. 	Let $(\Omega, \mu)$ be a finite measure space such that there are measurable subsets $\Omega_1, \dots, \Omega_n$   satisfying 
  	\begin{align*}
  		\Omega=\Omega_1\cup \cdots \cup \Omega_n, \quad \Omega_j \cap \Omega_k=\emptyset, \quad \forall 1\leq j,k \leq n, j \neq n.
  	\end{align*} 
  If $n\geq d$,	 then there exists a continuous ASF  $(\{f_\alpha\}_{\alpha\in \Omega},\{\tau_\alpha\}_{\alpha\in \Omega})$  for $\mathcal{X}$.
  \end{theorem}
\begin{proof}
Let $ (\{g_j \}_{j=1}^n, \{\omega_j \}_{j=1}^n) $ be an ASF  for $\mathcal{X}$	(they exist. Infact, any spanning collection can be turned to an ASF or a pair of basis for the space and its dual works, see \cite{KRISHNAJOHNSON, MAHESHTHESIS}). Define
\begin{align*}
&f_\alpha \coloneqq \frac{g_j}{\sqrt{\mu(\Omega_j)}}	,\quad  \forall \alpha \in \Omega_j, \forall 1\leq j \leq n, \quad \tau_\alpha \coloneqq \frac{\omega_j}{\sqrt{\mu(\Omega_j)}}	,\quad  \forall \alpha \in \Omega_j,\forall 1\leq j \leq n.
\end{align*} 
Then 
\begin{align*}
	\int_{\Omega} f_\alpha (x)\tau_\alpha\, d \mu (\alpha)&=\sum_{j=1}^n\int_{\Omega_j} f_\alpha (x)\tau_\alpha\, d \mu (\alpha)=\sum_{j=1}^n\int_{\Omega_j}\frac{g_j(x)}{\sqrt{\mu(\Omega_j)}}\frac{\omega_j}{\sqrt{\mu(\Omega_j)}}\, d \mu (\alpha)\\
	&=\sum_{j=1}^ng_j(x)\omega_j, \quad \forall x \in \mathcal{X}.
\end{align*}
Hence $(\{f_\alpha\}_{\alpha\in \Omega},\{\tau_\alpha\}_{\alpha\in \Omega})$ is a continuous ASF  for $\mathcal{X}$.
\end{proof}


  Like discrete case, we can get trace of frame operator using continuous ASFs.
  \begin{theorem}\label{TRACEFORMULA}
  	Let $(\{f_\alpha\}_{\alpha\in \Omega}, \{\tau_\alpha\}_{\alpha\in \Omega})$ be a continuous approximate Bessel family  for $\mathcal{X}$. Then 
  	\begin{align*}
  		&	\text{Tra}(S_{f,\tau})=\int_\Omega f_\alpha (\tau_\alpha)\, d \mu(\alpha),\\
  		&	\text{Tra}(S_{f,\tau}^2)=\int_\Omega\int_\Omega f_\beta (\tau_\alpha)f_\alpha (\tau_\beta)\, d \mu(\alpha)\, d \mu(\beta).
  	\end{align*}
  \end{theorem}
  \begin{proof}
  	Let $\{\omega_j\}_{j=1}^d$ be any basis for $\mathcal{X}$, where $d$ is the dimension of $\mathcal{X}$. Let $\{\zeta_j\}_{j=1}^d$ be the dual basis associated with $\{\omega_j\}_{j=1}^d$. Then $S_{f,\tau}x=	\sum_{j=1}^n\zeta_j(x)(S_{f,\tau}\omega_j)$ which gives 
  	\begin{align*}
  	\text{Tra}(S_{f,\tau})&=		\sum_{j=1}^n\zeta_j(S_{f,\tau}\omega_j)=	\sum_{j=1}^n\zeta_j\left(\int_\Omega f_\alpha (\omega_j)\tau_\alpha\,d\mu(\alpha)\right)\\
  	&=\sum_{j=1}^n\int_\Omega f_\alpha (\omega_j)\zeta_j(\tau_\alpha)\,d\mu(\alpha)=\int_\Omega f_\alpha \left(\sum_{j=1}^n\zeta_j(\tau_\alpha)\omega_j\right)\,d\mu(\alpha)\\
  	&=\int_\Omega f_\alpha (\tau_\alpha)\, d \mu(\alpha)
  	\end{align*}
  and 
  	\begin{align*}
  	 	\text{Tra}(S^2_{f,\tau})&=		\sum_{j=1}^n\zeta_j(S^2_{f,\tau}\omega_j)=	\sum_{j=1}^n\zeta_j\left(\int_\Omega f_\alpha (S_{f,\tau}\omega_j)\tau_\alpha\,d\mu(\alpha)\right)\\
  	&=\sum_{j=1}^n\int_\Omega f_\alpha (S_{f,\tau}\omega_j)\zeta_j(\tau_\alpha)\,d\mu(\alpha)=\int_\Omega f_\alpha \left(\sum_{j=1}^n\zeta_j(\tau_\alpha)S_{f,\tau}\omega_j\right)\,d\mu(\alpha)\\
  	&=\int_\Omega f_\alpha (S_{f,\tau}\tau_\alpha)\, d \mu(\alpha)=\int_\Omega f_\alpha\left(\int_{\Omega}f_\beta (\tau_\alpha)\tau_\beta \, d \mu (\beta)\right)\, d \mu(\alpha)\\
  	&=\int_\Omega\int_\Omega f_\beta (\tau_\alpha)f_\alpha (\tau_\beta)\, d \mu(\alpha)\, d \mu(\beta).
  	\end{align*}
  \end{proof}
  Note that we did not assume   any condition on set of vectors and functionals to derive Theorem \ref{DFBS}. The reason is that frame operator always exists in discrete case. To get  the existence of frame operator we assumed Besselness in Theorem \ref{TRACEFORMULA}. We now derive continuous versions of Theorem  \ref{FIRSTWELCHBANACH}   and Theorem \ref{WELCHBANACH}.
 \begin{theorem}\label{FIRSTORDERCONTINUOUS}
	\textbf{(First order  continuous Welch bound for Banach spaces)}  Let $(\Omega, \mu)$ be a $\sigma$-finite measure space and $(\{f_\alpha\}_{\alpha\in \Omega}, \{\tau_\alpha\}_{\alpha\in \Omega})$ be a 	 continuous Bessel  family for finite dimensional Banach space $\mathcal{X}$ of dimension $d$. If the diagonal $\Delta\coloneqq \{(\alpha, \alpha):\alpha \in \Omega\}$ is measurable in the measure space $\Omega\times \Omega$,
	\begin{align*}
		\int_{\Omega\times\Omega}|f_\alpha(\tau_\beta)f_\beta(\tau_\alpha)|\, d(\mu\times\mu)(\alpha,\beta)<\infty,
	\end{align*} 
and the operator $	S_{f, \tau}:\mathcal{X}\ni x \mapsto S_{f, \tau}x\coloneqq \int_{\Omega}f_\alpha(x)\tau_\alpha \, d\mu (\alpha) \in\mathcal{X}$  is diagonalizable and its eigenvalues are all non negative,  then 
	\begin{align}\label{1}
  		\int_{\Omega\times\Omega}f_\alpha(\tau_\beta)f_\beta(\tau_\alpha)\, d(\mu\times\mu)(\alpha,\beta)&=\int_{\Omega}\int_{\Omega}f_\alpha(\tau_\beta)f_\beta(\tau_\alpha)\, d \mu(\alpha)\, d \mu(\beta)\nonumber\\
  		&\geq \frac{1}{d}	\left(\int_{\Omega}
  		f_\alpha(\tau_\alpha)\, d \mu(\alpha)\right)^2=	\frac{1}{\text{dim}(\mathcal{X})}\left(\int_{\Omega}
  		f_\alpha(\tau_\alpha)\, d \mu(\alpha)\right)^2.
  	\end{align}	
  	and 
  \begin{align*}
  		\sup _{\alpha, \beta \in \Omega, \alpha\neq \beta}|f_\alpha(\tau_\beta)f_\beta(\tau_\alpha)|&\geq \frac{\frac{1}{d}	\left(\int_{\Omega}
  			f_\alpha(\tau_\alpha)\, d \mu(\alpha)\right)^2-	\int_{\Delta}|f_\alpha(\tau_\alpha)|^2\, d(\mu\times\mu)(\alpha,\alpha)}{(\mu\times\mu)((\Omega\times\Omega)\setminus\Delta)}\\
  		&=\frac{\frac{1}{\text{dim}(\mathcal{X})}	\left(\int_{\Omega}
  			f_\alpha(\tau_\alpha)\, d \mu(\alpha)\right)^2-	\int_{\Delta}|f_\alpha(\tau_\alpha)|^2\, d(\mu\times\mu)(\alpha,\alpha)}{(\mu\times\mu)((\Omega\times\Omega)\setminus\Delta)},
  \end{align*}
  \begin{align*}
  	\sup _{\alpha, \beta \in \Omega, \alpha\neq \beta}|f_\alpha(\tau_\beta)|&\geq \sqrt{\frac{\frac{1}{d}	\left(\int_{\Omega}
  			f_\alpha(\tau_\alpha)\, d \mu(\alpha)\right)^2-	\int_{\Delta}|f_\alpha(\tau_\alpha)|^2\, d(\mu\times\mu)(\alpha,\alpha)}{(\mu\times\mu)((\Omega\times\Omega)\setminus\Delta)}}\\
  		&=\sqrt{\frac{\frac{1}{\text{dim}(\mathcal{X})}	\left(\int_{\Omega}
  				f_\alpha(\tau_\alpha)\, d \mu(\alpha)\right)^2-	\int_{\Delta}|f_\alpha(\tau_\alpha)|^2\, d(\mu\times\mu)(\alpha,\alpha)}{(\mu\times\mu)((\Omega\times\Omega)\setminus\Delta)}}.
  \end{align*}
  Further, equality holds in Inequality (\ref{1}) if and only if $ (\{f_\alpha\}_{\alpha\in \Omega}, \{\tau_\alpha\}_{\alpha\in \Omega}) $ is  a tight continuous  ASF for $\mathcal{X}$. 
  In particular, if $f_\alpha(\tau_\alpha)=1$ for all $\alpha \in \Omega$, then 
  \begin{align*}
  	\sup _{\alpha, \beta \in \Omega, \alpha\neq \beta}|f_\alpha(\tau_\beta)f_\beta(\tau_\alpha)|&\geq \frac{1}{(\mu\times\mu)((\Omega\times\Omega)\setminus\Delta)}\left[\frac{\mu(\Omega)^2}{d}-(\mu\times\mu)(\Delta)\right]\\
  	&=\frac{1}{(\mu\times\mu)((\Omega\times\Omega)\setminus\Delta)}\left[\frac{\mu(\Omega)^2}{\text{dim}(\mathcal{X})}-(\mu\times\mu)(\Delta)\right]. 
  \end{align*}
  and we have \textbf{first order continuous Welch bound for Banach spaces}
  \begin{align*}
  	 	\sup _{\alpha, \beta \in \Omega, \alpha\neq \beta}|f_\alpha(\tau_\beta)|&\geq \sqrt{\frac{1}{(\mu\times\mu)((\Omega\times\Omega)\setminus\Delta)}\left[\frac{\mu(\Omega)^2}{d}-(\mu\times\mu)(\Delta)\right]}\\
  	 	&= \sqrt{\frac{1}{(\mu\times\mu)((\Omega\times\Omega)\setminus\Delta)}\left[\frac{\mu(\Omega)^2}{\text{dim}(\mathcal{X})}-(\mu\times\mu)(\Delta)\right]}.
  \end{align*}
  \end{theorem}
  \begin{proof}
  	Let $\lambda_1, \dots, \lambda_d$ be eigenvalues of the frame operator $S_{f, \tau}$. Then $\lambda_1, \dots, \lambda_d\geq0$. Now   using  Theorem  \ref{TRACEFORMULA} we get 
   	\begin{align*}
  	\left(\int_\Omega f_\alpha (\tau_\alpha)\, d \mu(\alpha)\right)^2&=	(\operatorname{Tra}(S_{f,\tau}))^2=\left(\sum_{k=1}^d
  	\lambda_k\right)^2\leq d \sum_{k=1}^d
  	\lambda_k^2\\
  	&=d\operatorname{Tra}(S^2_{f,\tau})=d\int_\Omega\int_\Omega f_\beta (\tau_\alpha)f_\alpha (\tau_\beta)\, d \mu(\alpha)\, d \mu(\beta).
  \end{align*}	
  	which gives the first inequality.	
  	Using Fubini's theorem, now we get 
  	
  	 \begin{align*}
  		\frac{1}{d} \left(\int_\Omega f_\alpha (\tau_\alpha)\, d \mu(\alpha)\right)^2&\leq\int_\Omega\int_\Omega f_\alpha (\tau_\beta)f_\beta (\tau_\alpha)\, d \mu(\alpha)\, d \mu(\beta)
  		=	\int_{\Omega\times\Omega}f_\alpha(\tau_\beta)f_\beta(\tau_\alpha)\, d(\mu\times\mu)(\alpha,\beta)\\
  		&=	\int_{\Delta}f_\alpha(\tau_\beta)f_\beta(\tau_\alpha)\, d(\mu\times\mu)(\alpha,\beta)+	\int_{(\Omega\times\Omega)\setminus \Delta}f_\alpha(\tau_\beta)f_\beta(\tau_\alpha)\, d(\mu\times\mu)(\alpha,\beta)\\
  		&=\int_{\Delta}|f_\alpha(\tau_\alpha)|^2\, d(\mu\times\mu)(\alpha,\alpha)+	\int_{(\Omega\times\Omega)\setminus \Delta}f_\alpha(\tau_\beta)f_\beta(\tau_\alpha)\, d(\mu\times\mu)(\alpha,\beta)\\
  		&\leq \int_{\Delta}|f_\alpha(\tau_\alpha)|^2\, d(\mu\times\mu)(\alpha,\alpha)+(\mu\times\mu)((\Omega\times\Omega)\setminus\Delta)\sup _{\alpha, \beta \in \Omega, \alpha\neq \beta}|f_\alpha(\tau_\beta)f_\beta(\tau_\alpha)|.	
  	\end{align*}	
  \end{proof}
  \begin{theorem}\label{CONTINUOUSWELCHMAINSECOND}
  	\textbf{(Continuous Welch bounds for Banach spaces)} Let $m \in \mathbb{N}$,  $(\Omega, \mu)$ be a $\sigma$-finite measure space and $(\{f_\alpha\}_{\alpha\in \Omega}, \{\tau_\alpha\}_{\alpha\in \Omega})$ be a 	 continuous Bessel  family for finite dimensional Banach space $\mathcal{X}$ of dimension $d$. If the diagonal $\Delta\coloneqq \{(\alpha, \alpha):\alpha \in \Omega\}$ is measurable in the measure space $\Omega\times \Omega$,
  \begin{align*}
  	\int_{\Omega\times\Omega}|f_\alpha(\tau_\beta)f_\beta(\tau_\alpha)|^m\, d(\mu\times\mu)(\alpha,\beta)<\infty,
  \end{align*} 
  and the operator $ S_{f,\tau}	: \text{Sym}^m(\mathcal{X})\ni x \mapsto  \int_{\Omega}f_\alpha^{\otimes m}(x)\tau_\alpha ^{\otimes m}\, d\mu (\alpha)\in\text{Sym}^m(\mathcal{X})$  is diagonalizable and its eigenvalues are all non negative,  then 
  \begin{align}\label{2}
  	\int_{\Omega\times\Omega}f_\alpha(\tau_\beta)^mf_\beta(\tau_\alpha)^m\, d(\mu\times\mu)(\alpha,\beta)&=\int_{\Omega}\int_{\Omega}f_\alpha(\tau_\beta)^mf_\beta(\tau_\alpha)^m\, d \mu(\alpha)\, d \mu(\beta)\nonumber\\
  	&\geq \frac{1}{{d+m-1 \choose m}}	\left(\int_{\Omega}
  	f_\alpha(\tau_\alpha)^m\, d \mu(\alpha)\right)^2 \nonumber \\
  	&=	\frac{1}{\text{dim}(\text{Sym}^m(\mathcal{X}))}\left(\int_{\Omega}
  	f_\alpha(\tau_\alpha)^m\, d \mu(\alpha)\right)^2.
  \end{align}	
  and 
  \begin{align}\label{WELCHCONTINUOUS4}
  	\sup _{\alpha, \beta \in \Omega, \alpha\neq \beta}|f_\alpha(\tau_\beta)f_\beta(\tau_\alpha)|^m&\geq \frac{\frac{1}{{d+m-1 \choose m}}	\left(\int_{\Omega}
  		f_\alpha(\tau_\alpha)^m\, d \mu(\alpha)\right)^2-	\int_{\Delta}|f_\alpha(\tau_\alpha)|^{2m}\, d(\mu\times\mu)(\alpha,\alpha)}{(\mu\times\mu)((\Omega\times\Omega)\setminus\Delta)}\nonumber\\
  	&=\frac{\frac{1}{\text{dim}(\text{Sym}^m(\mathcal{X}))}	\left(\int_{\Omega}
  		f_\alpha(\tau_\alpha)^m\, d \mu(\alpha)\right)^2-	\int_{\Delta}|f_\alpha(\tau_\alpha)|^{2m}\, d(\mu\times\mu)(\alpha,\alpha)}{(\mu\times\mu)((\Omega\times\Omega)\setminus\Delta)},
  \end{align}
  \begin{align}\label{WELCHCONTINUOUS5}
  	\sup _{\alpha, \beta \in \Omega, \alpha\neq \beta}|f_\alpha(\tau_\beta)|^m&\geq \sqrt{\frac{\frac{1}{{d+m-1 \choose m}}	\left(\int_{\Omega}
  			f_\alpha(\tau_\alpha)^m\, d \mu(\alpha)\right)^2-	\int_{\Delta}|f_\alpha(\tau_\alpha)|^{2m}\, d(\mu\times\mu)(\alpha,\alpha)}{(\mu\times\mu)((\Omega\times\Omega)\setminus\Delta)}}\nonumber\\
  		&=\sqrt{\frac{\frac{1}{\text{dim}(\text{Sym}^m(\mathcal{X}))}	\left(\int_{\Omega}
  				f_\alpha(\tau_\alpha)^m\, d \mu(\alpha)\right)^2-	\int_{\Delta}|f_\alpha(\tau_\alpha)|^{2m}\, d(\mu\times\mu)(\alpha,\alpha)}{(\mu\times\mu)((\Omega\times\Omega)\setminus\Delta)}}.
  \end{align}
  Further, equality holds in Inequality (\ref{2}) if and only if $ (\{f_\alpha\}_{\alpha\in \Omega}, \{\tau_\alpha\}_{\alpha\in \Omega}) $ is  a tight continuous  ASF for $\text{Sym}^m(\mathcal{X})$. 
  In particular, if $f_\alpha(\tau_\alpha)=1$ for all $\alpha \in \Omega$, then 
  \begin{align*}
  	\sup _{\alpha, \beta \in \Omega, \alpha\neq \beta}|f_\alpha(\tau_\beta)f_\beta(\tau_\alpha)|^m&\geq \frac{1}{(\mu\times\mu)((\Omega\times\Omega)\setminus\Delta)}\left[\frac{\mu(\Omega)^2}{{d+m-1 \choose m}}-(\mu\times\mu)(\Delta)\right]\\
  	&=\frac{1}{(\mu\times\mu)((\Omega\times\Omega)\setminus\Delta)}\left[\frac{\mu(\Omega)^2}{\text{dim}(\text{Sym}^m(\mathcal{X}))}-(\mu\times\mu)(\Delta)\right]. 
  \end{align*}
  and we have \textbf{higher order continuous Welch bound for Banach spaces}
  \begin{align*}
  	\sup _{\alpha, \beta \in \Omega, \alpha\neq \beta}|f_\alpha(\tau_\beta)|^m&\geq \sqrt{\frac{1}{(\mu\times\mu)((\Omega\times\Omega)\setminus\Delta)}\left[\frac{\mu(\Omega)^2}{{d+m-1 \choose m}}-(\mu\times\mu)(\Delta)\right]}\\
  	&= \sqrt{\frac{1}{(\mu\times\mu)((\Omega\times\Omega)\setminus\Delta)}\left[\frac{\mu(\Omega)^2}{\text{dim}(\text{Sym}^m(\mathcal{X}))}-(\mu\times\mu)(\Delta)\right]}.
  \end{align*}	
  \end{theorem}
 \begin{proof}
Let $\lambda_1, \dots, \lambda_{\text{dim}(\text{Sym}^m(\mathcal{X}))}$ be eigenvalues of $	S_{f, \tau}$. Then 

\begin{align*}
	\left(\int_{\Omega}
	f_\alpha(\tau_\alpha)^m\, d \mu(\alpha)\right)^2	&=\left(\int_{\Omega}
	f_\alpha^{\otimes m}(\tau_\alpha^{\otimes m})\, d \mu(\alpha)\right)^2=	(\operatorname{Tra}(S_{f,\tau}))^2=\left(\sum_{l=1}^{\text{dim}(\text{Sym}^m(\mathcal{X}))}
	\lambda_l\right)^2\\
	&\leq 
	\text{dim}(\text{Sym}^m(\mathcal{X})) \sum_{l=1}^{\text{dim}(\text{Sym}^m(\mathcal{X}))}
	\lambda_l^2
	={d+m-1 \choose m}\operatorname{Tra}(S^2_{f,\tau})\\
	&={d+m-1 \choose m}\int_{\Omega}\int_{\Omega}f_\alpha^{\otimes m}(\tau_\beta^{\otimes m})f_\beta^{\otimes m}(\tau_\alpha^{\otimes m})\, d \mu(\alpha)\, d \mu(\beta)
	\\
	&={d+m-1 \choose m}\int_{\Omega}\int_{\Omega}f_\alpha(\tau_\beta)^mf_\beta(\tau_\alpha)^m\, d \mu(\alpha)\, d \mu(\beta).
\end{align*} 	
and 
 \begin{align*}
	&\frac{1}{{d+m-1 \choose m}}  \left(\int_{\Omega}
	f_\alpha(\tau_\alpha)^m\, d \mu(\alpha)\right)^2=\int_{\Omega}\int_{\Omega}f_\alpha(\tau_\beta)^mf_\beta(\tau_\alpha)^m\, d \mu(\alpha)\, d \mu(\beta)\\
	&=	\int_{\Omega\times\Omega}f_\alpha(\tau_\beta)^mf_\beta(\tau_\alpha)^m\, d(\mu\times\mu)(\alpha,\beta)\\
	&=\int_{(\Omega\times\Omega)\setminus \Delta}f_\alpha(\tau_\beta)^mf_\beta(\tau_\alpha)^m\, d(\mu\times\mu)(\alpha,\beta)+\int_{\Delta}f_\alpha(\tau_\alpha)^{2m}\, d(\mu\times\mu)(\alpha,\alpha)\\
	&\leq \int_{(\Omega\times\Omega)\setminus \Delta}|f_\alpha(\tau_\beta)f_\beta(\tau_\alpha)|^m\, d(\mu\times\mu)(\alpha,\beta)+\int_{\Delta}|f_\alpha(\tau_\alpha)|^{2m}\, d(\mu\times\mu)(\alpha,\alpha)\\
	&\leq (\mu\times\mu)((\Omega\times\Omega)\setminus\Delta) 	\sup _{\alpha, \beta \in \Omega, \alpha\neq \beta}|f_\alpha(\tau_\beta)f_\beta(\tau_\alpha)|^m+\int_{\Delta}|f_\alpha(\tau_\alpha)|^{2m}\, d(\mu\times\mu)(\alpha,\alpha).
\end{align*}
 \end{proof}
 \begin{corollary}
  	Theorem  \ref{WELCHBANACH}    is a corollary of Theorem 	\ref{CONTINUOUSWELCHMAINSECOND}.
  \end{corollary}
  \begin{proof}
  	Take $\Omega=\{1,\dots,n\} $ and $\mu$ as the counting measure.
  \end{proof}
 \begin{theorem}
 Theorem 	\ref{CONTINUOUSHILBERT} is a corollary of Theorem 	\ref{CONTINUOUSWELCHMAINSECOND}.
 \end{theorem}
\begin{proof}
Let $\{\tau_\alpha\}_{\alpha\in \Omega}$ be a 	normalized continuous Bessel  family for $\mathcal{H}$ of dimension $d$. We define 
\begin{align*}
	f_\alpha:\mathcal{H}\ni h \mapsto \langle h, \tau_\alpha \rangle \in \mathbb{K}, \quad \forall \alpha \in \Omega
\end{align*}
and apply 	Theorem 	\ref{CONTINUOUSWELCHMAINSECOND}.
\end{proof} 
  As  observed in \cite{MAHESHKRISHNA}, we again observe  that given a measure space $\Omega$, the diagonal $\Delta$ need not be measurable (see \cite{DRAVECKY}).  
  \begin{question}
  	\textbf{Classify measure spaces $(\Omega, \mu) $ such that Theorem \ref{CONTINUOUSWELCHMAINSECOND}  holds?} In other words, \textbf{given a measure space $(\Omega, \mu) $, does the validity of Inequality (\ref{2}) or Inequality (\ref{WELCHCONTINUOUS4})  or Inequality (\ref{WELCHCONTINUOUS5}) implies  conditions on meausre space $(\Omega, \mu) $, say $\sigma$-finite?}
  \end{question}
 Now we derive continuous versions of Theorem   \ref{DISCRETEARBITRARY1}  and Theorem \ref{DISCRETEARBITRARY2}.
  \begin{theorem}
  	Let $(\Omega, \mu)$ be a $\sigma$-finite measure space and $(\{f_\alpha\}_{\alpha\in \Omega}, \{\tau_\alpha\}_{\alpha\in \Omega})$ be a 	 continuous Bessel  family for finite dimensional Banach space $\mathcal{X}$ of dimension $d$. If the diagonal $\Delta\coloneqq \{(\alpha, \alpha):\alpha \in \Omega\}$ is measurable in the measure space $\Omega\times \Omega$
  	and the operator $	S_{f, \tau}:\mathcal{X}\ni x \mapsto S_{f, \tau}x\coloneqq \int_{\Omega}f_\alpha(x)\tau_\alpha \, d\mu (\alpha) \in\mathcal{X}$  is diagonalizable and its eigenvalues are all non negative,  then 
  	\begin{align*}
  		\frac{1}{d}	\operatorname{Tra}(S_{f,\tau})^r\geq 	\left(\frac{1}{d}\int_\Omega  f_\alpha(\tau_\alpha)\, d \mu(\alpha)\right)^r ,\quad \forall r \in [1, \infty)
  	\end{align*}	
  	and 
  	\begin{align*}
  		\frac{1}{d}	\operatorname{Tra}(S_{f,\tau})^r\leq 	\left(\frac{1}{d}\int_\Omega  f_\alpha(\tau_\alpha)\, d \mu(\alpha)\right)^r ,\quad \forall r \in (0,1).
  	\end{align*}
  In particular, if $f_\alpha(\tau_\alpha)=1$ for all $\alpha \in \Omega$, then 
  	\begin{align*}
  	\frac{1}{\mu(\Omega)}	\operatorname{Tra}(S_{f,\tau})^r\geq \left(\frac{\mu(\Omega)}{d}\right)^{r-1} ,\quad \forall r \in [1, \infty)
  \end{align*}	
  and 
  \begin{align*}
  	\frac{1}{\mu(\Omega)}	\operatorname{Tra}(S_{f,\tau})^r\leq \left(\frac{\mu(\Omega)}{d}\right)^{r-1} ,\quad \forall r \in (0,1).
  \end{align*}
  \end{theorem}
  \begin{proof}
  	Let $\lambda_1, \dots, \lambda_{d}$ be eigenvalues of $	S_{f,\tau}$. Let $r \in [1, \infty)$. Jensen's inequality	gives 
  	\begin{align*}
  		\left(\frac{1}{d}\sum_{k=1}^d\lambda_k\right)^r\leq \frac{1}{d}\sum_{k=1}^d\lambda_k^r.
  	\end{align*}
  	Therefore
  	\begin{align*}
  		\left(\frac{1}{d}\int_\Omega  f_\alpha(\tau_\alpha)\, d \mu(\alpha)\right)^r=	\left(\frac{1}{d}\operatorname{Tra}(S_{f,\tau})\right)^r\leq \frac{1}{d}\operatorname{Tra}(	S_{f,\tau})^r.
  	\end{align*}
  	Similarly the case $ r \in (0,1)  $ follows by using Jensen's inequality.
  \end{proof}
  \begin{theorem}\label{PWELCH}
  	Let $2<p<\infty$,  $(\Omega, \mu)$ be a $\sigma$-finite measure space and $(\{f_\alpha\}_{\alpha\in \Omega}, \{\tau_\alpha\}_{\alpha\in \Omega})$ be a 	 continuous Bessel  family for finite dimensional Banach space $\mathcal{X}$ of dimension $d$ such that $f_\alpha(\tau_\alpha)=1$ for all $\alpha \in \Omega$. If the diagonal $\Delta\coloneqq \{(\alpha, \alpha):\alpha \in \Omega\}$ is measurable in the measure space $\Omega\times \Omega$,
  	\begin{align*}
  		\int_{\Omega\times\Omega}|f_\alpha(\tau_\beta)f_\beta(\tau_\alpha)|\, d(\mu\times\mu)(\alpha,\beta)<\infty,
  	\end{align*} 
  	and the operator $	S_{f, \tau}:\mathcal{X}\ni x \mapsto S_{f, \tau}x\coloneqq \int_{\Omega}f_\alpha(x)\tau_\alpha \, d\mu (\alpha) \in\mathcal{X}$ is diagonalizable and its eigenvalues are all non negative,  then 
  	 	\begin{align*}
  		\int_{\Omega\times\Omega}|f_\alpha(\tau_\beta)f_\beta(\tau_\alpha)|^\frac{p}{2}\, d(\mu\times\mu)(\alpha,\beta)&=	\int_{\Omega}\int_{\Omega}|f_\alpha(\tau_\beta)f_\beta(\tau_\alpha)|^\frac{p}{2}\, d \mu(\alpha)\, d \mu(\beta)\\
  		&\geq  	\frac{1}{(\mu\times \mu)((\Omega\times\Omega)\setminus\Delta)^{\frac{p}{2}-1}}	\left(\frac{\mu(\Omega)^2}{d}-(\mu\times\mu)(\Delta)\right)^\frac{p}{2}+(\mu\times\mu)(\Delta).
  	\end{align*}
  \end{theorem}
  \begin{proof}
  	Define $r\coloneqq 2p/(p-2)$ and $q$ be the conjugate index of $p/2$. Then $q=r/2$. Using Theorem \ref{FIRSTORDERCONTINUOUS} and Holder's inequality, we  have 
  	
  	\begin{align*}
  		\frac{\mu(\Omega)^2}{d}-(\mu\times\mu)(\Delta)&\leq	\int_{(\Omega\times\Omega)\setminus\Delta}|f_\alpha(\tau_\beta)f_\beta(\tau_\alpha)|\, d(\mu\times\mu)(\alpha,\beta)\\
  		&\leq \left(\int_{(\Omega\times\Omega)\setminus\Delta}|f_\alpha(\tau_\beta)f_\beta(\tau_\alpha)|^\frac{p}{2}\, d(\mu\times\mu)(\alpha,\beta)\right)^\frac{2}{p}\left(\int_{(\Omega\times\Omega)\setminus\Delta}d(\mu\times\mu)(\alpha,\beta)\right)^\frac{1}{q}\\
  		&=\left(\int_{(\Omega\times\Omega)\setminus\Delta}|f_\alpha(\tau_\beta)f_\beta(\tau_\alpha)|^\frac{p}{2}\, d(\mu\times\mu)(\alpha,\beta)\right)^\frac{2}{p}(\mu\times \mu)((\Omega\times\Omega)\setminus\Delta)^\frac{1}{q}\\	
  		&=\left(\int_{(\Omega\times\Omega)\setminus\Delta}|f_\alpha(\tau_\beta)f_\beta(\tau_\alpha)|^\frac{p}{2}\, d(\mu\times\mu)(\alpha,\beta)\right)^\frac{2}{p}(\mu\times \mu)((\Omega\times\Omega)\setminus\Delta)^\frac{2}{r}\\	
  		&=\left(\int_{(\Omega\times\Omega)\setminus\Delta}|f_\alpha(\tau_\beta)f_\beta(\tau_\alpha)|^\frac{p}{2}\, d(\mu\times\mu)(\alpha,\beta)\right)^\frac{2}{p}(\mu\times \mu)((\Omega\times\Omega)\setminus\Delta)^\frac{p-2}{p}	
  	\end{align*}
  	which gives 
  	\begin{align*}
  		\left(\frac{\mu(\Omega)^2}{d}-(\mu\times\mu)(\Delta)\right)^\frac{p}{2}\leq \left(\int_{(\Omega\times\Omega)\setminus\Delta}|f_\alpha(\tau_\beta)f_\beta(\tau_\alpha)|^\frac{p}{2}\, d(\mu\times\mu)(\alpha,\beta)\right)(\mu\times \mu)((\Omega\times\Omega)\setminus\Delta)^{\frac{p}{2}-1}.
  	\end{align*}
  	Therefore 
  	\begin{align*}
  		&\frac{1}{(\mu\times \mu)((\Omega\times\Omega)\setminus\Delta)^{\frac{p}{2}-1}}	\left(\frac{\mu(\Omega)^2}{d}-(\mu\times\mu)(\Delta)\right)^\frac{p}{2}+(\mu\times\mu)(\Delta)\\
  		&~=\frac{1}{(\mu\times \mu)((\Omega\times\Omega)\setminus\Delta)^{\frac{p}{2}-1}}	\left(\frac{\mu(\Omega)^2}{d}-(\mu\times\mu)(\Delta)\right)^\frac{p}{2}+\int_{\Delta}|f_\alpha(\tau_\beta)f_\beta(\tau_\alpha)|^\frac{p}{2}\, d(\mu\times\mu)(\alpha,\beta)\\
  		&~\leq \int_{(\Omega\times\Omega)\setminus\Delta}|f_\alpha(\tau_\beta)f_\beta(\tau_\alpha)|^\frac{p}{2}\, d(\mu\times\mu)(\alpha,\beta)+\int_{\Delta}|f_\alpha(\tau_\alpha)f_\alpha(\tau_\alpha)|^\frac{p}{2}\, d(\mu\times\mu)(\alpha,\alpha)\\
  		&~=\int_{\Omega\times\Omega}|f_\alpha(\tau_\beta)f_\beta(\tau_\alpha)|^\frac{p}{2}\, d(\mu\times\mu)(\alpha,\beta).
  	\end{align*}
  \end{proof}
   In the continuous case, Theorem \ref{LEVENSTEINBOUND}   leads to the following problem.
  \begin{question}
 \textbf{ Whether there is a  continuous version of Theorem \ref{LEVENSTEINBOUND} for Banach spaces?. In particular, does there exists a continuous version of }
  	\begin{enumerate}[\upshape(i)]
  		\item  \textbf{ Bukh-Cox bound for Banach spaces?}
  		\item  \textbf{ Orthoplex/Rankin bound for Banach spaces?}
  		\item  \textbf{ Levenstein bound for Banach spaces?}
  		\item  \textbf{ Exponential bound for Banach spaces?}
  	\end{enumerate}		
  \end{question}

\section{Applications of continuous Welch bounds for Banach spaces}
 Here we list continuous versions of corresponding concepts, results and open problems stated in   Section \ref{SECTIONTWO}.  Throughout this section, $(\Omega, \mu)$ is a $\sigma$-finite measure space. We further assume that the diagonal $\Delta$ is measurable.
 \begin{definition}\label{CRMSBANACH}
 Let  $ (\{f_\alpha\}_{\alpha\in \Omega}, \{\tau_\alpha\}_{\alpha\in \Omega})$ be a 	 continuous Bessel  family for  a finite dimensional Banach space $\mathcal{X}$ of dimension $d$   satisfying  $f_\alpha(\tau_\alpha)=1$ for all $\alpha \in \Omega$. Assume that  the  frame operator 
 	$	S_{f, \tau}$  is diagonalizable and its eigenvalues are all non negative. Also assume that 
 	\begin{align*}
 		\int_{\Omega\times\Omega}|f_\alpha(\tau_\beta)f_\beta(\tau_\alpha)|\, d(\mu\times\mu)(\alpha,\beta)<\infty.
 	\end{align*} 
 	  We define the \textbf{ continuous root-mean-square (RMS) absolute cross relation} of  $ (\{f_\alpha\}_{\alpha\in \Omega}, \{\tau_\alpha\}_{\alpha\in \Omega})$   as 
 	\begin{align*}
 		I_{\text{CRMS}} (\{f_\alpha\}_{\alpha\in \Omega}, \{\tau_\alpha\}_{\alpha\in \Omega})\coloneqq \left(\frac{1}{(\mu\times\mu)((\Omega\times\Omega)\setminus\Delta)}\int_{(\Omega\times\Omega)\setminus\Delta}f_\alpha(\tau_\beta)f_\beta(\tau_\alpha)\, d(\mu\times\mu)(\alpha,\beta)\right)^\frac{1}{2}.
 	\end{align*}
 \end{definition}
	 \begin{theorem}
 	Let $ (\{f_\alpha\}_{\alpha\in \Omega}, \{\tau_\alpha\}_{\alpha\in \Omega})$ be as in Definition \ref{CRMSBANACH}. Then
 	\begin{align*}
 	\sup_{\alpha, \beta \in \Omega, \alpha\neq \beta}|f_\alpha(\tau_\beta)| \geq 	I_{\text{CRMS}} (\{f_\alpha\}_{\alpha\in \Omega}, \{\tau_\alpha\}_{\alpha\in \Omega})\geq \left(\frac{1}{(\mu\times\mu)((\Omega\times\Omega)\setminus\Delta)}\left[\frac{\mu(\Omega)^2}{d}-(\mu\times\mu)(\Delta)\right]\right)^\frac{1}{2}.
 	\end{align*}
 \end{theorem}
 \begin{definition}\label{CPFP}
 	 Let  $ (\{f_\alpha\}_{\alpha\in \Omega}, \{\tau_\alpha\}_{\alpha\in \Omega})$ be a 	 continuous Bessel  family for  a finite dimensional Banach space $\mathcal{X}$ of dimension $d$   satisfying  $f_\alpha(\tau_\alpha)=1$ for all $\alpha \in \Omega$. Assume that  the  frame operator 
 	$	S_{f, \tau}$  is diagonalizable and its eigenvalues are all non negative. Also assume that 
 	\begin{align*}
 		\int_{\Omega\times\Omega}|f_\alpha(\tau_\beta)f_\beta(\tau_\alpha)|\, d(\mu\times\mu)(\alpha,\beta)<\infty.
 	\end{align*}  We define the \textbf{continuous pseudo frame potential} of $  (\{f_\alpha\}_{\alpha\in \Omega}, \{\tau_\alpha\}_{\alpha\in \Omega}) $ as 
 	\begin{align*}
 		CPFP  (\{f_\alpha\}_{\alpha\in \Omega}, \{\tau_\alpha\}_{\alpha\in \Omega})\coloneqq  	\int_{\Omega\times\Omega}f_\alpha(\tau_\beta)f_\beta(\tau_\alpha)\, d(\mu\times\mu)(\alpha,\beta).
 	\end{align*}
 \end{definition}
\begin{theorem}\label{CPFPTHEOREM}
 	Let $ (\{f_\alpha\}_{\alpha\in \Omega}, \{\tau_\alpha\}_{\alpha\in \Omega})$ be as in Definition \ref{CPFP}. Then 	
 	\begin{align}\label{CPFPINENQUALITY}
 		\mu(\Omega)^2	\sup_{\alpha, \beta \in \Omega}|f_\alpha(\tau_\beta)| ^2\geq 	CPFP (\{f_\alpha\}_{\alpha\in \Omega}, \{\tau_\alpha\}_{\alpha\in \Omega})	\geq 		\frac{\mu(\Omega)^2}{d}.
 	\end{align}
 \end{theorem}
 
 \begin{question}
 	\textbf{Is  there  a characterization of $ (\{f_\alpha\}_{\alpha\in \Omega}, \{\tau_\alpha\}_{\alpha\in \Omega}) $ by using equality in Inequality (\ref{CPFPINENQUALITY})?}
 \end{question}

 \begin{definition}
 Let  $ (\{f_\alpha\}_{\alpha\in \Omega}, \{\tau_\alpha\}_{\alpha\in \Omega}) $ be a continuous ASF for $\mathcal{X}$ satisfying $f_\alpha(\tau_\alpha)=1$ for all $\alpha \in \Omega$. Assume that  the  frame operator 
 $	S_{f, \tau}$  is diagonalizable and its eigenvalues are all non negative. Also assume that 
 \begin{align*}
 	\int_{\Omega\times\Omega}|f_\alpha(\tau_\beta)f_\beta(\tau_\alpha)|\, d(\mu\times\mu)(\alpha,\beta)<\infty.
 \end{align*} 
	We define the  \textbf{ continuous frame correlation} of $ (\{f_\alpha\}_{\alpha\in \Omega}, \{\tau_\alpha\}_{\alpha\in \Omega})$ as 
 	\begin{align*}
 		\mathcal{M}(\{f_\alpha\}_{\alpha\in \Omega}, \{\tau_\alpha\}_{\alpha\in \Omega})\coloneqq 	\sup_{\alpha, \beta \in \Omega, \alpha\neq \beta}|f_\alpha(\tau_\beta)|.
 	\end{align*}
 \end{definition}
  
 \begin{definition}\label{CONTINUOUSGRASSMANNIANDEFINITION}
 	Let  $ (\{f_\alpha\}_{\alpha\in \Omega}, \{\tau_\alpha\}_{\alpha\in \Omega}) $ be a continuous ASF for $\mathcal{X}$ satisfying $\|f_\alpha\|=1, \|\tau_\alpha\|=1, f_\alpha(\tau_\alpha)=1$ for all $\alpha \in \Omega$. Assume that  the  frame operator 
 	$	S_{f, \tau}$  is diagonalizable and its eigenvalues are all non negative. Also assume that 
 	\begin{align*}
 		\int_{\Omega\times\Omega}|f_\alpha(\tau_\beta)f_\beta(\tau_\alpha)|\, d(\mu\times\mu)(\alpha,\beta)<\infty.
 	\end{align*} 
  Continuous  ASF $ (\{f_\alpha\}_{\alpha\in \Omega}, \{\tau_\alpha\}_{\alpha\in \Omega}) $ is said to be a \textbf{continuous Grassmannian frame} for  $\mathcal{X}$ if 
 	\begin{align*}
 		\mathcal{M}(\{f_\alpha\}_{\alpha\in \Omega}, \{\tau_\alpha\}_{\alpha\in \Omega})=\inf\bigg\{&\mathcal{M}(\{g_\alpha\}_{\alpha\in \Omega}, \{\omega_\alpha\}_{\alpha\in \Omega}): (\{g_\alpha\}_{\alpha\in \Omega}, \{\omega_\alpha\}_{\alpha\in \Omega}) \text{ is a continuous  ASF for }\mathcal{X}  \\
 		& \text{satisfying } \|g_\alpha\|=1, \|\omega_\alpha\|=1, g_\alpha(\omega_\alpha)=1, \forall \alpha \in \Omega \text{, the frame operator }  S_{g,\omega}\\
 		& \text{is diagonalizable and its eigenvalues are all non negative and }\\
 		&	\int_{\Omega\times\Omega}|g_\alpha(\omega_\beta)g_\beta(\omega_\alpha)|\, d(\mu\times\mu)(\alpha,\beta)<\infty\bigg\}.		
 	\end{align*}		
 \end{definition}
 \begin{question}
	\textbf{Classify measure spaces and (finite dimensional) Banach spaces so that continuous Grassmannian ASFs exist}.
\end{question}
\begin{definition}
 Let  $ (\{f_\alpha\}_{\alpha\in \Omega}, \{\tau_\alpha\}_{\alpha\in \Omega}) $ be a continuous ASF for $\mathcal{X}$ satisfying $f_\alpha(\tau_\alpha)=1$ for all $\alpha \in \Omega$. Assume that  the  frame operator 
 $	S_{f, \tau}$  is diagonalizable and its eigenvalues are all non negative. Continuous  ASF $ (\{f_\alpha\}_{\alpha\in \Omega}, \{\tau_\alpha\}_{\alpha\in \Omega}) $ is said to be 	\textbf{$\gamma$-equiangular} if there exists $\gamma\geq0$ such that
 	\begin{align*}
 		|f_\alpha(\tau_\beta)|=\gamma, \quad \forall  \alpha, \beta \in \Omega, \alpha\neq \beta.
 	\end{align*}
 \end{definition}

 \begin{conjecture}
 	(\textbf{Continuous Zauner's conjecture for Banach spaces}) \textbf{For a given measure space  $(\Omega, \mu)$ and for  every $d\in \mathbb{N}$, there exists a $\gamma$-equiangular  continuous ASF $(\{f_\alpha\}_{\alpha\in \Omega}, \{\tau_\alpha\}_{\alpha\in \Omega})$  for $\mathbb{C}^d$ such that $\mu(\Omega)=d^2$}.
 \end{conjecture}

 \begin{theorem}\label{CNG3}
 	Let  $ (\{f_\alpha\}_{\alpha\in \Omega}, \{\tau_\alpha\}_{\alpha\in \Omega}) $ be as in Definition \ref{CONTINUOUSGRASSMANNIANDEFINITION}. Then 
 	\begin{align}\label{CEQUIANGULARINEQUALITY}
 		\mathcal{M}(\{f_\alpha\}_{\alpha\in \Omega}, \{\tau_\alpha\}_{\alpha\in \Omega}) \geq \left(\frac{1}{(\mu\times\mu)((\Omega\times\Omega)\setminus\Delta)}\left[\frac{\mu(\Omega)^2}{d}-(\mu\times\mu)(\Delta)\right]\right)^\frac{1}{2}\eqqcolon\gamma.
 	\end{align}
 	If the continuous ASF is $\gamma$-equiangular, then we have equality in Inequality (\ref{CEQUIANGULARINEQUALITY}).
 \end{theorem}
 \begin{question}
 	\begin{enumerate}[\upshape (i)]
 		\item \textbf{	Whether the equality in Inequality (\ref{CEQUIANGULARINEQUALITY}) implies that the continuous ASF $ (\{f_\alpha\}_{\alpha\in \Omega}, \{\tau_\alpha\}_{\alpha\in \Omega})$  is $\gamma$-equiangular?}
 		\item \textbf{Whether the validity of  Inequality (\ref{CEQUIANGULARINEQUALITY}) implies   there is a relation between the measure of $\Omega$ and dimension of the space (like Theorem 2.3 in \cite{STROHMERHEATH})?}
 	\end{enumerate}
 \end{question}

 \bibliographystyle{plain}
 \bibliography{reference.bib}

\end{document}